\documentclass[11pt,leqno]{article}
\usepackage{amsmath,amsthm,amsfonts,amssymb}
\usepackage[colorlinks=true, pdfstartview=FitV, linkcolor=blue, citecolor=red, urlcolor=blue]{hyperref}
\hypersetup{breaklinks=true}
\usepackage{graphicx}
\usepackage{caption,subfig}
\usepackage{color}
\usepackage{tikz}
\usetikzlibrary{calc}
\usepackage{calrsfs}

\usepackage{geometry}
\newcommand{\R}{{\mathbb R}}
\newcommand{\C}{{\mathbb C}}
\newcommand{\U}{{\mathbb U}}
\newcommand{\F}{{\mathbb F}}
\newcommand{\eps}{{\varepsilon}}
\newcommand{\bfS}{{\bf S}}

\newtheorem{theorem}{Theorem}[section]
\newtheorem{proposition}[theorem]{Proposition}

\newtheorem{lem}[theorem]{Lemma}

\numberwithin{equation}{section}

\title{On the Mach stem configuration\\
with shallow angle}

\author{Jean-Fran\c{c}ois {\sc Coulombel}\thanks{Institut de Math\'ematiques de Toulouse - UMR 5219, 
Universit\'e de Toulouse ; CNRS, Universit\'e Paul Sabatier, 118 route de Narbonne, 31062 Toulouse Cedex 9 , France. 
Part of this work was achieved while J.-F. C. was a member of the Laboratoire de Math\'ematiques Jean Leray in Nantes, 
and the author gratefully thanks this institution for the stimulating working conditions it has provided. Research of J.-F. C. 
was supported by ANR project BoND, ANR-13-BS01-0009, and ANR project Nabuco, ANR-17-CE40-0025. Email: 
{\tt jean-francois.coulombel@math.univ-toulouse.fr}} $\,$ \& 
Mark {\sc Williams}\thanks{University of North Carolina, Mathematics Department, CB 3250, Phillips Hall, 
Chapel Hill, NC 27599. USA. Email: {\tt williams@email.unc.edu}.}}

\begin{document}

\maketitle

\begin{abstract}
The aim of this article is to explain why similar \emph{weak stability} criteria appear in both the construction of \emph{steady} Mach 
stem configurations bifurcating from a reference planar shock wave solution to the compressible Euler equations, as studied by Majda 
and Rosales [Stud. Appl. Math. 1984], and in the weakly nonlinear stability analysis of the same planar shock performed by the same 
authors [SIAM J. Appl. Math. 1983], when that shock is viewed as a solution to the \emph{evolutionary} compressible Euler equations. 
By carefully studying the normal mode analysis of planar shocks in the evolutionary case, we show that for a uniquely defined tangential 
velocity with respect to the planar front, the temporal frequency which allows for the amplification of highly oscillating wave packets, 
when reflected on the shock front, vanishes. This specific tangential velocity is found to coincide with the expression given by Majda 
and Rosales [Stud. Appl. Math. 1984] when they determine the steady planar shocks that admit arbitrarily close steady Mach stem 
configurations. The links between the causality conditions for Mach stems and the so-called Lopatinskii determinant for shock waves 
are also clarified.
\end{abstract}

\section{Introduction}
\label{intro}

The stability analysis of shock waves in gas dynamics now has a very long history, dating back to pioneering works for instance by D'yakov 
\cite{dyakov}, Erpenbeck \cite{erpenbeck} and followers, see, e.g., \cite{swanfowles,BE} and references therein. At the linearized level, 
determining stability amounts to finding unstable and/or neutrally stable eigenmodes. The most favorable situation corresponds to nonexistence 
of unstable nor neutrally stable eigenmodes. Following Majda's memoirs \cite{M1,M2}, this regime will be referred to as that of \emph{uniform 
stability}. When unstable eigenmodes (of positive real part) occur, \emph{violent instability} is expected to take place. We shall mainly be 
concerned here with the intermediate situation where unstable eigenmodes do not occur but neutrally stable (that is, purely imaginary) 
eigenmodes do arise. This regime will be referred to as that of \emph{weak stability}. From a more mathematical point of view, the regime 
we shall consider corresponds to the so-called WR (for \emph{Weak Real}) class identified in \cite{BRSZ} but we shall only be concerned 
here with the particular problem of shock waves in gas dynamics.

In the first part \cite{MR1} of a series of papers, Majda and Rosales considered the regime of weak stability for reacting shocks and identified 
some weakly nonlinear waves that exhibited an \emph{amplification} phenomenon. The weakly nonlinear waves considered in \cite{MR1} are 
approximate solutions to the \emph{evolutionary} compressible Euler equations that are small, high frequency perturbations of a planar reference 
shock with \emph{zero} tangential velocity. The analysis in \cite{MR1} has been recently generalized by the authors in \cite{CW}. In the second 
part \cite{MR2} of their series, Majda and Rosales considered the \emph{steady} Euler equations in two space dimensions and proved that 
\emph{in the exact same regime of weak stability}, and for some specific \emph{nonzero} tangential velocity, a step shock could `bifurcate' 
into a family of \emph{steady} Mach stems with shallow angle (see Figures \ref{fig:planarshock} and \ref{fig:machstem} hereafter for an 
illustration), thus providing ``a completely independent confirmation of that theory'' (quote from \cite{MR2}). The links between the two problems 
in \cite{MR1,MR2} and the appearance of the exact same weak stability condition are somehow hidden in lengthy calculations. It is the purpose 
of this article to clearly explain the role of the nonzero tangential velocity in \cite{MR2} and its connection with the so-called Lopatinskii determinant 
that has now been repeatedly computed for decades in the stability analysis of planar shock waves with \emph{zero} tangential velocity. As should 
be clear from our analysis below, the problems studied in \cite{MR1} and \cite{MR2} are not so independent as they might look at first glance.

Our main conclusions can be summarized as follows : the tangential velocity exhibited in \cite{MR2} for the analysis of step shocks bifurcating into 
steady Mach stems with shallow angle coincides with the root to the Lopatinskii determinant encoding the stability of the planar shock with same 
density/entropy/normal velocity but \emph{zero} tangential velocity. This first result is proved in Proposition \ref{prop1} below. Now a crucial property 
of the Euler equations is Galilean invariance which will reflect here into the result of Lemma \ref{determinant-velocity} that connects the Lopatinskii 
determinant associated with two shock waves that share the same density/entropy/normal velocity but do not have the same tangential velocity. 
Combining the results of Proposition \ref{prop1} and Lemma \ref{determinant-velocity}, we find that the tangential velocity exhibited in \cite{MR2} 
is equivalently determined by the requirement that the associated Lopatinskii determinant vanishes at the time frequency \emph{zero}, hence the 
connection with the \emph{steady} Euler equations. We review some of the arguments in \cite{MR2} and explain why the \emph{causality} 
conditions used there as an admissibility criterion for Mach stems turn out to yield the same causality conditions used in \cite{MR1} to discriminate 
between incoming and outgoing wave packets.

The article is organized as follows. In Section \ref{normal}, we briefly review the normal mode analysis which yields uniform/weak stability criteria 
for shock waves in \emph{evolutionary} gas dynamics. In the weak stability regime, we verify that the expression of the tangential velocity given in 
\cite{MR2} coincides with the root to the Lopatinskii determinant. This first observation, based on `brute force computations', is explained with further 
details in Section \ref{sect:MachStem} where we prove that the bifurcation problem for the steady Euler equations considered in \cite{MR2} amounts 
to determining a tangential velocity for which the Lopatinskii determinant vanishes at the time frequency zero. We also give a complete construction 
of the family of Mach stems bifurcating from the reference planar shock, thus completing and clarifying some of the arguments in \cite{MR2}. In 
particular, we make precise the assumptions on the pressure law under which the Mach stem construction can be achieved.

\section{The normal mode analysis of shock waves}
\label{normal}

This Section could deal with any space dimension $d \ge 2$, but since Section \ref{sect:MachStem} will deal specifically with 
two-dimensional flows, we restrict from now on to $d=2$ in order to keep the same notation throughout the whole article. We 
follow the presentation in \cite{mplohr} and consider a compressible inviscid fluid endowed with a \emph{complete} equation of 
state $e=e(\tau,s)$. Here $\tau$ denotes the specific volume of the fluid, $s$ denotes the specific entropy and $e$ denotes the 
specific internal energy. The pressure $p$ and temperature $T$ are defined by the fundamental law of thermodynamics
\begin{equation*}
{\rm d}e \, = \, -p \, {\rm d}\tau +T \, {\rm d}s \, .
\end{equation*}
The (evolutionary) compressible Euler equations in two space dimensions are written in the compact form
\begin{equation}
\label{euler}
\partial_t f_0(U) +\partial_{x_1} f_1(U) +\partial_{x_2} f_2(U) \, = \, 0 \, ,
\end{equation}
where $U=(\tau,{\bf u},s)$ is the four component vector\footnote{Vectors are written either as rows or columns when no confusion 
is possible.} of unknowns and ${\bf u}=(u,v) \in \R^2$ is the fluid velocity. The fluxes $f_\alpha$, $\alpha=0,1,2$, in \eqref{euler} are 
given by
\begin{align}
&f_0(U) \, := \, \begin{bmatrix}
\rho \\
\rho \, u \\
\rho \, v \\
\dfrac{1}{2} \, \rho \, |{\bf u}|^2 +\rho \, e \end{bmatrix} \, , \, 
&f_1(U) \, := \, \begin{bmatrix}
\rho \, u \\
\rho \, u^2 +p \\
\rho \, u \, v \\
\left( \dfrac{1}{2} \, \rho \, |{\bf u}|^2 +\rho \, e +p \right) u \end{bmatrix} \, ,\label{deff} \\
& &f_2(U) \, := \, \begin{bmatrix}
\rho \, v \\
\rho \, u \, v \\
\rho \, v^2 +p \\
\left( \dfrac{1}{2} \, \rho \, |{\bf u}|^2 +\rho \, e +p \right) v \end{bmatrix} \, ,\notag
\end{align}
where $\rho :=1/\tau$ denotes the density. Our assumptions on the equation of state are (part of) the classical Bethe-Weyl 
inequalities (we refer again to \cite{mplohr}):
\begin{equation*}
p>0 \, ,\quad \quad T>0 \, ,\quad \quad \dfrac{\partial^2 e}{\partial \tau^2} >0 \, ,\quad \quad 
\dfrac{\partial^2 e}{\partial s \, \partial \tau} <0 \, ,\quad \quad \dfrac{\partial^3 e}{\partial^3 \tau} <0 \, .
\end{equation*}
We then define the sound speed $c$ and the so-called Gr\"uneisen coefficient $\Gamma$ by setting
\begin{equation*}
c^2 \, := \, \tau^2 \, \dfrac{\partial^2 e}{\partial \tau^2} \, = \, -\tau^2 \, \dfrac{\partial p}{\partial \tau} \, ,\quad \quad 
\Gamma \, := \, -\dfrac{\tau}{T} \, \dfrac{\partial^2 e}{\partial s \, \partial \tau} \, = \, \dfrac{\tau}{T} \, \dfrac{\partial p}{\partial s} \, ,
\end{equation*}
both being positive quantities.

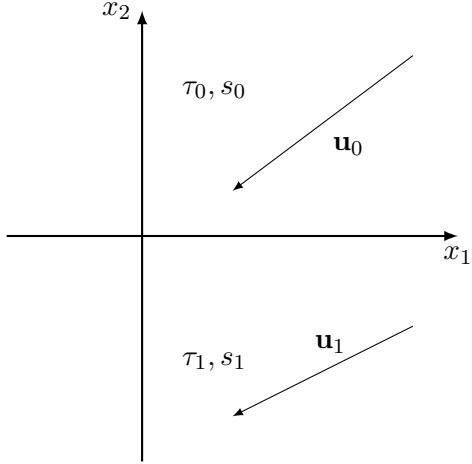
\begin{figure}[h!]
\begin{center}
\begin{tikzpicture}[scale=1.2,>=latex]
\draw[thick,->] (-1.5,0) -- (3.5,0) node[below] {$x_1$};
\draw[thick,->] (0,-2.5)--(0,2.5) node[left] {$x_2$};
\draw[->] (3,2) -- (1,0.5);
\node[right] at (2,1) {${\bf u}_0$};
\node[above] at (0.8,1.4) {$\tau_0,s_0$};
\draw[->] (3,-1) -- (1,-2);
\node[right] at (1.8,-1.2) {${\bf u}_1$};
\node[above] at (0.8,-1.6) {$\tau_1,s_1$};
\end{tikzpicture}
\caption{The steady planar shock.}
\label{fig:planarshock}
\end{center}
\end{figure}

A shock wave\footnote{The problem studied in \cite{MR1} deals with reacting flows, for which the pressure laws ahead 
and behind the shock do not necessarily coincide. The normal mode analysis in \cite{MR1} though is independent of the 
presence of a chemical reaction as long as Lax shock inequalities \eqref{ineglax} below are satisfied. We thus restrict 
to the more standard framework of the compressible Euler equations without reaction and with same pressure law on either 
side of the shock for simplicity.} is a piecewise constant solution to \eqref{euler} satisfying Lax shock inequalities \cite{Lax}. 
In the context of the Euler equations \eqref{euler}, one can always perform a Galilean change of frame and rotate the coordinate 
axes so that the shock wave is steady and reads
\begin{equation}
\label{shock}
\underline{U} \, = \, \begin{cases}
\underline{U}_0 \, := \, (\tau_0,{\bf u}_0,s_0) & \text{\rm if } x_2>0 \, ,\\
\underline{U}_1 \, := \, (\tau_1,{\bf u}_1,s_1) & \text{\rm if } x_2<0 \, ,
\end{cases}
\end{equation}
as depicted in Figure \ref{fig:planarshock}. The Rankine-Hugoniot conditions, which ensure that \eqref{shock} is a weak solution 
to \eqref{euler}, and Lax shock inequalities then read
\begin{subequations}
\label{RH}
\begin{align}
& {\bf j} \, := \, -\rho_0 \, v_0 \, = \, -\rho_1 \, v_1 >0 \, , \quad & & u_0 \, = \, u_1 \, =: \, \overline{u} \, , \, \label{RH1} \\
& {\bf j}^2 \, (\tau_1 -\tau_0) \, = \, p_0 -p_1 \, , \quad & & e_1-e_0 +\dfrac{p_1+p_0}{2} \, (\tau_1-\tau_0) \, = \, 0 \, ,\label{RH2} \\
& 0 < \dfrac{-v_1}{c_1} < 1 < \dfrac{-v_0}{c_0} \, . & & \label{ineglax}
\end{align}
\end{subequations}
The tangential velocity $\overline{u}$ could also be set to zero by a Galilean change of frame, but \emph{we do not do so} here 
in order to highlight the links with the analysis of Section \ref{sect:MachStem}. Observe that up to changing $x_1$ into $-x_1$ 
and the tangential velocity accordingly, we can always assume $\overline{u} \le 0$ without loss of generality. This is the convention 
in \cite{MR2} and we follow it here.

The (linear) stability properties of the particular solution \eqref{shock} have been made precise after a long series of contributions 
which we have partly recalled in the introduction. The analysis is based on a normal mode decomposition that is briefly summarized 
below. We refer to the appendix of \cite{Zumbrun-JL} and to \cite[chapter 15]{BS} for a detailed and complete analysis of this stability 
problem, and just introduce the notation that will be useful later on for our purpose.

We introduce the Fourier variable $\eta \in \R$ dual to $x_1$, and the Laplace variable $z=\delta -i\, \gamma$, $\gamma \ge 0$, 
dual to $t$. We focus on the state `$1$' behind the shock since there are no \emph{stable} modes ahead of the shock. The 
eigenmodes $\exp (i\, \omega \, x_2)$ under consideration correspond to complex numbers $\omega$ of nonpositive imaginary 
part\footnote{Here we have extracted the $i$ factor from all frequency parameters, and consider the half space $x_2 \le 0$, which 
is the reason why we consider Im $z \le 0$ and Im $\omega \le 0$. The conventions in \cite{Zumbrun-JL,BS} are different but passing 
from one to the other is harmless.} such that there exists a nontrivial solution $(\dot{\tau},\dot{u},\dot{v},\dot{s}) \in \C^4$ to the linear 
system
\begin{equation*}
\begin{cases}
(z+\overline{u} \, \eta +v_1 \, \omega) \, \dot{\tau} -\tau_1 \, \eta \, \dot{u} -\tau_1 \, \omega \, \dot{v} \, = \, 0 \, ,& \\[1ex]
(z+\overline{u} \, \eta +v_1 \, \omega) \, \dot{u} -\dfrac{c_1^2}{\tau_1} \, \eta \, \dot{\tau} 
+\Gamma_1 \, T_1 \, \eta \, \dot{s} \, = \, 0 \, ,& \\[1.5ex]
(z+\overline{u} \, \eta +v_1 \, \omega) \, \dot{v} -\dfrac{c_1^2}{\tau_1} \, \omega \, \dot{\tau} 
+\Gamma_1 \, T_1 \, \omega \, \dot{s} \, = \, 0 \, ,& \\[1ex]
(z+\overline{u} \, \eta +v_1 \, \omega) \, \dot{s} \, = \, 0 \, . &
\end{cases}
\end{equation*}
All quantities with a `$1$' index are evaluated behind the shock, that is, in the region $\{ x_2<0 \}$ for \eqref{shock}. The analysis of 
the latter linear system gives rise to two eigenmodes $\omega_0$ and $\omega_-$, which are defined by
\begin{equation}
\label{eigenmodes}
z+\overline{u} \, \eta +v_1 \, \omega_0 \, = \, 0 \, ,\quad \quad 
(z+\overline{u} \, \eta +v_1 \, \omega_-)^2 \, = \, c_1^2 \, (\eta^2 +\omega_-^2) \, ,
\end{equation}
where the choice of $\omega_-$ in \eqref{eigenmodes} is such that $\omega_-$ has negative imaginary part when $z$ also 
has negative imaginary part, and $\omega_-$ is extended continuously up to real values of $z$, see \cite[chapters 14 \& 15]{BS}. 
The corresponding eigenspaces associated with $\omega_0$ and $\omega_-$ are
\begin{equation}
\label{eigenspaces}
E_0(z,\eta) \, := \, \text{\rm Span } \left\{ \, 
\begin{bmatrix}
0 \\
\, \omega_0 \, \\
\, -\eta \, \\
0 \end{bmatrix} \, , \, \begin{bmatrix}
\, \Gamma_1 \, T_1 \, \tau_1 \, \\
0 \\
0 \\
c_1^2 \end{bmatrix} \, \right\} \, ,\quad \quad 
E_-(z,\eta) \, := \, \text{\rm Span} \begin{bmatrix}
\, \tau_1 \, (z+\overline{u} \, \eta +v_1 \, \omega_-) \, \\
c_1^2 \, \eta \\
c_1^2 \, \omega_- \\
0 \end{bmatrix} \, .
\end{equation}
Specifying to the case $z=0$, $\eta=1$, we obtain
\begin{equation}
\label{eigenspaces01}
E_0(0,1) \, = \, \text{\rm Span } \left\{ 
\begin{bmatrix}
0 \\
\overline{u} \\
\, v_1 \, \\
0 \end{bmatrix} \, , \, \begin{bmatrix}
\, \Gamma_1 \, T_1 \, \tau_1 \, \\
0 \\
0 \\
c_1^2 \end{bmatrix} \right\} \, ,\quad 
E_-(0,1) \, = \, \text{\rm Span} \begin{bmatrix}
\, \tau_1 \, (\overline{u}+v_1 \, \omega_-(0,1)) \, \\
c_1^2 \\
c_1^2 \, \omega_-(0,1) \\
0 \end{bmatrix} \, ,
\end{equation}
with
\begin{equation}
\label{eigenmode01}
\omega_-(0,1) \, = \, \dfrac{1}{c_1^2-v_1^2} \, \left( 
v_1 \, \overline{u} +c_1 \, \text{\rm sgn} \, (\overline{u}) \, \sqrt{\overline{u}^2+v_1^2-c_1^2} \right) \, .
\end{equation}
The expression \eqref{eigenmode01} is valid as long as the tangential velocity $\overline{u}$ satisfies $\overline{u}^2+v_1^2>c_1^2$ 
(here sgn denotes the sign function). In that case, $(0,1)$ is a \emph{hyperbolic} frequency for the linearized Euler equations in 
$\{ x_2<0 \}$ because all the roots to \eqref{eigenmodes} are real and they depend smoothly on $(z,\eta)$ near $(0,1)$. The reason 
for choosing $\text{\rm sgn} \, (\overline{u})$ in \eqref{eigenmode01} rather than the opposite sign is forced by the fact that $\omega_-$ 
should be continuous with respect to $(z,\eta)$ on the set $\{ \text{\rm Im} \, z \le 0 \, , \, \eta \in \R \}$. The determination of the appropriate 
sign in \eqref{eigenmode01} can also be interpreted as a \emph{causality} condition, meaning that oscillating wave packets associated with 
the phase
$$
0 \cdot t +1 \cdot x_1 +\omega_-(0,1) \, x_2
$$
should have a group velocity that points inside the half space $\{ x_2<0 \}$ (in other words, the second component of the group velocity 
should be negative). This selection criterion is used when computing the weakly nonlinear expansions in \cite{MR1}.

The subspace $E^s(z,\eta)$ of values $(\dot{\tau},\dot{\bf u},\dot{s})|_{x_2=0^-}$ under consideration is then the direct sum of $E_0(z,\eta)$ 
and $E_-(z,\eta)$, except when the eigenmodes $\omega_0$ and $\omega_-$ coincide. We shall mainly be concerned here with the frequency 
$(z,\eta)=(0,1)$ for which $\omega_0$ and $\omega_-$ do not coincide, and therefore refer to the above mentioned works for the precise 
decomposition of $E^s(z,\eta)$ whenever $\omega_0=\omega_-$. We thus consider
$$
E^s(z,\eta) \, := \, E_0(z,\eta) \oplus E_-(z,\eta) \, ,
$$
and refer from now on to $E^s(z,\eta)$ as the \emph{stable subspace}. (It is a three dimensional subspace of $\C^4$.)

The stability analysis of the step shock \eqref{shock} amounts to determining whether there exist some frequencies $(z,\eta)$ for 
which one can find a \emph{nonzero} pair $(\chi,\dot{U}) \in \C \times E^s(z,\eta)$ that satisfies the linearized Rankine-Hugoniot 
conditions which, following \cite[pages 426-427]{BS} with our notation, read\footnote{Here we have kept track of the possibly 
nonzero tangential velocity $\overline{u}$ while linearizing the Rankine-Hugoniot conditions and made some elementary linear 
combinations between several equations to get \eqref{linearRH}.}
\begin{equation}
\label{linearRH}
\begin{cases}
-\dfrac{v_1}{\tau_1^2} \, \dot{\tau} +\dfrac{1}{\tau_1} \, \dot{v} \, = \, -(\rho_1 -\rho_0) \, (z+\overline{u} \, \eta) \, \chi \, ,& \\[1.5ex]
\dfrac{v_1}{\tau_1} \, \dot{u} \, = \, -(p_1-p_0) \, \eta \, \chi \, ,& \\[1.5ex]
-\dfrac{v_1^2+c_1^2}{\tau_1^2} \, \dot{\tau} +2 \, \dfrac{v_1}{\tau_1} \, \dot{v} +\dfrac{\Gamma_1 \, T_1}{\tau_1} 
\, \dot{s} \, = \, 0 \, ,& \\[1.5ex]
-\dfrac{1}{2} \, \Big( \dfrac{c_1^2}{\tau_1^2} \, (\tau_1-\tau_0) +p_1-p_0 \Big) \, \dot{\tau} 
+T_1 \, \Big( 1+\dfrac{\Gamma_1 \, (\tau_1-\tau_0)}{2\, \tau_1} \Big) \, \dot{s} \, = \, 0 \, . &
\end{cases}
\end{equation}
For future use (we refer once again to \cite[chapter 15]{BS}), it should be kept in mind that \eqref{linearRH} is an equivalent formulation 
of the relation
\begin{equation}
\label{RHequiv}
{\rm d}f_2(\underline{U}_1) \, \dot {U} \, = \, -\chi \, \Big( z \, \big( f_0(\underline{U}_1) -f_0(\underline{U}_0) \big) 
+\eta \, \big( f_1(\underline{U}_1) -f_1(\underline{U}_0) \big) \Big) \, ,
\end{equation}
after some elementary manipulations on the rows of \eqref{RHequiv}. The Jacobian matrix  ${\rm d}f_2 (\underline{U}_1)$ is invertible 
because of Lax shock inequalities \eqref{ineglax} (its eigenvalues are $v_1-c_1$, $v_1$ and $v_1+c_1$), so any nonzero solution 
$(\chi,\dot{U}) \in \C \times E^s(z,\eta)$ to \eqref{linearRH} must satisfy $\chi \neq 0$. In other words, either \eqref{linearRH} has 
no nonzero solution in $\C \times E^s(z,\eta)$, or the set of solutions is a one-dimensional subspace of $\C \times E^s(z,\eta)$. 
Equivalently, the stability analysis of the step shock \eqref{shock} amounts to determining whether there holds
\begin{equation}
\label{RHequiv'}
{\rm d}f_2(\underline{U}_1)^{-1} \, \Big( z \, \big( f_0(\underline{U}_1) -f_0(\underline{U}_0) \big) 
+\eta \, \big( f_1(\underline{U}_1) -f_1(\underline{U}_0) \big) \Big) \in E^s(z,\eta) \, .
\end{equation}

The shock wave linear stability problem may be encoded as the determination of the roots of an appropriate determinant, which is 
usually referred to as the Lopatinskii determinant. More precisely, one can define a complex number $\Delta (\overline{u},z,\eta)$ 
such that $\Delta (\overline{u},z,\eta) =0$ if and only if there exists a nonzero $(\chi,\dot{U}) \in \C \times E^s(z,\eta)$ solution to 
\eqref{linearRH}. We have highlighted here the dependence of the determinant $\Delta$ on the tangential velocity $\overline{u}$, 
in order to state the following result, which is very simple though fundamental if one wants to understand the links between the 
calculations performed in \cite{MR1} and \cite{MR2}.

\begin{lem}
\label{determinant-velocity}
Let the Lopatinskii determinant $\Delta (\overline{u},z,\eta)$ be defined\footnote{As explained in \cite{BS}, there are many possible ways 
to define $\Delta$ but all possible definitions give rise to the same zeroes, if there are any.} such that $\Delta (\overline{u},z,\eta)=0$ if and 
only if there exists a nonzero $(\chi,\dot{U}) \in \C \times E^s(z,\eta)$ solution to \eqref{linearRH}. Then there holds
\begin{equation*}
\Delta (\overline{u},z,\eta) \, = \, \Delta (0,z+\overline{u} \, \eta,\eta) \, ,
\end{equation*}
for all complex number $z$ of nonnegative imaginary part and all real number $\eta$.
\end{lem}

\noindent The proof of Lemma \ref{determinant-velocity} is elementary and is based on inspection of \eqref{eigenmodes}, \eqref{eigenspaces} 
and \eqref{linearRH}. Namely, it follows from the above definitions that the change of parameters $\tilde{z} := z+\overline{u} \, \eta$ in 
\eqref{eigenmodes}, \eqref{eigenspaces} and \eqref{linearRH} does not affect the sign of the imaginary part of $z$ and it reduces the 
above analysis to the case $\overline{u}=0$. We leave this elementary verification to the reader.

Unsurprisingly, the existence/nonexistence of a zero to $\Delta$, and accordingly the sign of the imaginary part of $z$ to discriminate 
between violent instability and neutral stability, is independent of $\overline{u}$ (a consequence of Galilean invariance). However, the 
precise location of the zeroes to $\Delta$, provided that they exist, \emph{does depend} on $\overline{u}$.

The stability analysis of step shocks with \emph{zero} tangential velocity can be summarized as follows (see \cite{M1} and 
\cite[chapter 15]{BS}), where from now on $M_1 \in (0,1)$ denotes the Mach number $-v_1/c_1$ behind the shock:
\begin{itemize}
 \item If $M_1^2 \, (\tau_0/\tau_1-1)<1/(1+\Gamma_1)$, then $\Delta (0,z,\eta) \neq 0$ for any nonzero pair $(z,\eta) \in \C \times \R$ 
 with Im $z \le 0$. This regime corresponds to \emph{uniform stability}. In particular, the function $\Delta (0,\cdot,1)$ does not vanish 
 on the real line.
 
 \item If $M_1^2 \, (\tau_0/\tau_1-1)>(1+M_1)/\Gamma_1$, then $\Delta (0,z,\eta)=0$ for some pair $(z,\eta)$ with Im $z<0$ and 
 $\eta \neq 0$. This regime corresponds to \emph{violent instability}. Moreover, the function $\Delta (0,\cdot,1)$ does not vanish 
 on the real line.
 
 \item If
\begin{equation}
\label{weakstab}
\dfrac{1}{1+\Gamma_1} \, < \, M_1^2 \, \left( \dfrac{\tau_0}{\tau_1}-1 \right) \, < \, \dfrac{1+M_1}{\Gamma_1} \, ,
\end{equation}
 then there exists a uniquely determined velocity\footnote{We refer to $V$ as a velocity since it has the physical homogeneity of a velocity, 
 see equation \eqref{caracV}-{\rm (iii)}.} $V>0$ such that $\Delta (0,z,\eta)=0$ if and only if\footnote{The final conclusion of the shock 
 wave stability analysis is written here in two space dimensions. In higher space dimensions, one should rather write $z= \pm V \, |\eta|$ 
 instead of $z= \pm V \, \eta$ ($\eta$ becomes a real vector in dimension $d \ge 3$.)} $\eta \neq 0$ and $z= \pm V \, \eta$. This regime 
 corresponds to \emph{neutral stability} and falls into the so-called WR class of \cite{BRSZ}.
\end{itemize}

\noindent In the latter weakly stable case, the velocity $V>0$ can be characterized as follows, see \cite[Theorem 15.1]{BS}:
\begin{align}
{\rm (i)} & \quad V^2 \, > \, c_1^2 -v_1^2 \, ,\notag \\
{\rm (ii)} & \quad \left( 1+M_1^2 -M_1^2 \, \Gamma_1 \, \left( \dfrac{\tau_0}{\tau_1}-1 \right) \right) \, V^2 
\, < \, v_1^2 \, \big( 1-M_1^2 \big) \, \dfrac{\tau_0}{\tau_1} \, ,\label{caracV} \\
{\rm (iii)} & \quad \left( \big( k-1+M_1^2 \big) \, V^2 -v_1^2 \, \big( 1-M_1^2 \big) \, \dfrac{\tau_0}{\tau_1} \right)^2 
\, = \, k^2 \, V^2 \, \Big( M_1^2 \, V^2 -v_1^2 \, \big( 1-M_1^2 \big) \Big) \, , \notag
\end{align}
where the parameter $k$ in \eqref{caracV}-{\rm (iii)} is defined by:
$$
k \, := \, 2-M_1^2 \, \Gamma_1 \, \left( \dfrac{\tau_0}{\tau_1}-1 \right) \, .
$$
In the limit case
\begin{equation}
\label{weakstablim}
\dfrac{1}{1+\Gamma_1} \, = \, M_1^2 \, \left( \dfrac{\tau_0}{\tau_1}-1 \right) \, ,
\end{equation}
$V$ tends to the `glancing' value $\sqrt{c_1^2 -v_1^2}$. The other limit case
\begin{equation*}
M_1^2 \, \left( \dfrac{\tau_0}{\tau_1}-1 \right) \, = \, \dfrac{1+M_1}{\Gamma_1} \, ,
\end{equation*}
corresponds to a transition from weak stability to violent instability for which the Lopatinskii determinant vanishes at the frequency 
$(z,\eta)=(1,0)$. This case is studied in \cite{serreTAMS}.

For what follows, the key point to keep in mind is that if the function $\Delta (0,\cdot,1)$ vanishes on the real line, then either the shock 
wave \eqref{shock} satisfies \eqref{weakstab} or it satisfies one of two the limit cases of \eqref{weakstab}, namely \eqref{weakstablim}. 
In the other limit case of \eqref{weakstab}, the Lopatinskii determinant vanishes for some real number $z$ but with a zero tangential 
frequency $\eta$.

Our main result in this Section asserts that the tangential velocity exhibited in \cite{MR2} for giving rise to a bifurcation from a rectilinear 
shock to a family of steady Mach stems coincides (up to the sign convention) with the root $V$ to the Lopatinskii determinant in the 
regime \eqref{weakstab}. This is a preliminary connection between the shock wave stability analysis, that deals with an \emph{evolutionary} 
problem, and the \emph{steady} Mach stem bifurcation problem explored in \cite{MR2}.

\begin{proposition}
\label{prop1}
Assume that the inequalities \eqref{weakstab} are satisfied. Then the tangential velocity $c_\star$ defined in \cite[page 124]{MR2} 
coincides with the velocity $V$ for which the Lopatinskii determinant $\Delta (0,\cdot,\cdot)$ vanishes at $(\pm V,1)$. Thanks to 
Lemma \ref{determinant-velocity}, the tangential velocity $\overline{u} :=-c_\star$ in \cite{MR2} thus satisfies $\Delta (\overline{u},0,1)=0$.
\end{proposition}

\begin{proof}[Proof of Proposition \ref{prop1}]
We check the items {\rm (i)}, {\rm (ii)} and {\rm (iii)} in \eqref{caracV} one by one. For the sake of completeness, let us first recall how 
the tangential velocity $c_\star$ is defined in \cite[page 124]{MR2}. Assuming that the inequalities \eqref{weakstab} are satisfied, one 
first computes the (unique real) root $\Phi \in (M_1,1)$ to the second order polynomial equation
\begin{equation}
\label{eqPhi}
\Phi^2 +M_1 \, \Gamma_1 \, \Phi -1-\Gamma_1 +\dfrac{1-M_1^2}{(\tau_0/\tau_1-1) \, M_1^2} \, = \, 0 \, .
\end{equation}
Then one defines $\beta \in (-1,M_1)$ and $c_\star>0$ by the relations:
\begin{equation}
\label{defbetac}
\beta \, := \, \dfrac{2\, M_1 -(1+M_1^2) \, \Phi}{1+M_1^2 -2\, M_1 \, \Phi} \, ,\quad 
c_\star \, := \, c_1 \, \dfrac{1-M_1 \, \beta}{\sqrt{1-\beta^2}} \, .
\end{equation}
We are therefore going to show that the velocity $c_\star$ given in \eqref{defbetac} satisfies the characterization in \eqref{caracV}. 
For the sake of simplicity, we introduce the notation
\begin{equation}
\label{defnu}
\nu \, := \, \dfrac{\tau_0}{\tau_1} -1 \, > \, 0 \, ,
\end{equation}
which measures the compression ratio of the shock, the inequality $\nu>0$ following here from \eqref{weakstab} (anyway under rather 
mild additional assumptions on the pressure law, shock waves are always compressive discontinuities, see \cite{mplohr}).
\bigskip

$\bullet$ Verifying \eqref{caracV}-(i): one uses the definition \eqref{defbetac} of $\beta$ and computes
\begin{equation*}
c_\star^2 \, = \, c_1^2 \, \dfrac{(1-M_1 \, \beta)^2}{1-\beta^2} \, = \, c_1^2 \, \dfrac{(1-M_1 \, \Phi)^2}{1-\Phi^2} 
> c_1^2 \, \big( 1-M_1^2 \big) \, = \, c_1^2 -v_1^2 \, ,
\end{equation*}
where the final inequality follows from $\Phi \in (M_1,1)$.
\bigskip

$\bullet$ Verifying \eqref{caracV}-(ii): for future use, we introduce the rescaled parameter:
\begin{equation}
\label{defy}
y \, := \, M_1 \, \Phi \in (M_1^2,M_1) \, .
\end{equation}
Since $\Phi$ satisfies \eqref{eqPhi}, $y$ in \eqref{defy} is a root (and it is even the largest one) to the polynomial $Q$ that is 
defined by:
\begin{equation}
\label{defq1}
Q(Y) \, := \, Y^2 +M_1^2 \, \Gamma_1 \, Y -M_1^2 \, (1+\Gamma_1) +\dfrac{1-M_1^2}{\nu} \, .
\end{equation}

Using the relation
\begin{equation}
\label{relation1}
c_\star^2 \, = \, c_1^2 \, \dfrac{(1-M_1 \, \Phi)^2}{1-\Phi^2} \, = \, v_1^2 \, \dfrac{(1-y)^2}{M_1^2-y^2} \, ,
\end{equation}
which we have found at the previous step, and using also the above definition \eqref{defnu} of the parameter $\nu$, verifying 
\eqref{caracV}-(ii) amounts to proving the inequality
\begin{equation*}
\big( 1+M_1^2 -M_1^2 \, \Gamma_1 \, \nu \big) \, (1-y)^2 - \big( 1-M_1^2 \big) \, (1+\nu) \, \big( M_1^2 -y^2 \big) \, < \, 0 \, .
\end{equation*}
We now use the fact that $y$ is a root to the polynomial $Q$ in \eqref{defq1}, and write equivalently the latter inequality as
\begin{equation}
\label{ineqy1}
y \, \big( 2-M_1^2 \, \Gamma_1 \, \nu \big) \, \big( 1+M_1^2 +M_1^2 \, \Gamma_1 \big) 
\, > \, \dfrac{1}{\nu} \, \Big\{ \underbrace{-2\, \big( 1 -M_1^2 \big) +\nu \, M_1^2 \, \big( 4 +\big( 3-M_1^2 \big) \, \Gamma_1 
\big) -\nu^2 \, M_1^4 \, \Gamma_1 \, (2+\Gamma_1)}_{= \, (2-M_1^2 \, \Gamma_1 \, \nu) \, 
(M_1^2 -1+\nu \, M_1^2 \, (2+\Gamma_1))} \Big\} \, .
\end{equation}
From the inequalities \eqref{weakstab}, we know that $2-M_1^2 \, \Gamma_1 \, \nu$ is positive, and therefore, proving \eqref{ineqy1} 
amounts to showing
\begin{equation}
\label{ineqy2}
y \, > \, \dfrac{M_1^2 -1+\nu \, M_1^2 \, (2+\Gamma_1)}{\nu \, \big( 1+M_1^2 +M_1^2 \, \Gamma_1 \big)} \, ,
\end{equation}
where, recalling the definition \eqref{defy}, we know that $y$ is the positive and therefore largest root to the second degree polynomial 
$Q$ in \eqref{defq1}. The conclusion then follows from the relation
\begin{equation*}
Q \left( \dfrac{M_1^2-1 +\nu \, M_1^2 \, (2+\Gamma_1)}{\nu \, (1+M_1^2 +M_1^2 \, \Gamma_1)} \right) 
\, = \, \dfrac{ \big( 1-M_1^2 \big)^2 \, (1+\nu)}{\nu^2 \, (1+M_1^2 +M_1^2 \, \Gamma_1)^2} \, \big( 1-\nu \, M_1^2 \, (1+\Gamma_1) \big) 
\, < \, 0 \, ,
\end{equation*}
where the final inequality is a consequence of \eqref{weakstab}. This proves that the quantity on the right hand side in \eqref{ineqy2} 
lies between the two roots of $Q$, so \eqref{ineqy2} is satisfied. Rewinding the arguments above, we have verified \eqref{caracV}-(ii).
\bigskip

$\bullet$ Verifying \eqref{caracV}-(iii): factorizing $1-M_1^2$, we first observe that the velocity $V$ satisfies \eqref{caracV}-(iii) 
if and only if
\begin{equation}
\label{racineVtilde}
\Big( (k-1)^2 -M_1^2 \Big) \, V^4 +\Big( (k-1)^2 +1-2\, M_1^2 -2\, \nu \, \big( k-1+M_1^2 \big) \Big) \, v_1^2 \, V^2 
+v_1^4 \, \big( 1-M_1^2 \big) \, (1+\nu)^2 \, = \, 0 \, ,
\end{equation}
with the parameter $k$ defined as in \eqref{caracV}-(iii). We are now going to verify that the velocity $c_\star$ in \eqref{defbetac} 
satisfies \eqref{racineVtilde}. We use the relation \eqref{relation1} and substitute this expression of $c_\star^2$ (recall that $y$ is 
defined in \eqref{defy} and is a root of the polynomial $Q$ defined in \eqref{defq1}). Multiplying by the obviously nonzero quantity 
$v_1^{-4} \, (M_1^2-y^2)^2$, we are reduced to showing the relation
\begin{multline}
\label{relation2}
\Big( \big( 1-\nu \, M_1^2 \, \Gamma_1 \big)^2 -M_1^2 \Big) \, (1-y)^4 
+\big( 1-M_1^2 \big) \, (1+\nu)^2 \, \big( M_1^2 -y^2 \big)^2 \\
+\big( M_1^2-y^2 \big) \, (1-y)^2 \, \Big( \big( 1-\nu \, M_1^2 \, \Gamma_1 \big)^2 +1 -2\, M_1^2 
-2\, \nu \, \big( 1 +M_1^2 -\nu \, M_1^2 \, \Gamma_1\big) \Big) \, = \, 0 \, .
\end{multline}
The left hand side of \eqref{relation2} can be factorized as
\begin{equation*}
\nu \, Q(y) \, \big( a_2\, y^2 +a_1 \, y +a_0 \big) \, ,
\end{equation*}
where $Q$ is defined in \eqref{defq1}, and
\begin{align*}
a_2 \, & := \, 4 +\nu \, \big( 1 -M_1^2 -2\, M_1^2 \, \Gamma_1 \big) \, ,\\
a_1 \, & := \, -4 \, \big( 1+M_1^2 \big) +\nu \, M_1^2 \, \Gamma_1 \, \big( 3+M_1^2 \big) \, ,\\
a_0 \, & := \, \big( 1+M_1^2 \big)^2 -\nu \, M_1^2 \, \big( 1-M_1^2 \big) 
-\nu \, M_1^2 \, \Gamma_1 \, \big( 1+M_1^2 \big) \, .
\end{align*}
Since $Q(y)$ is zero, the left hand side of \eqref{relation2} is indeed zero, as expected. This shows that the velocity $c_\star$ in 
\eqref{defbetac} satisfies all three conditions in \eqref{caracV} and therefore coincides with $V$.
\end{proof}

Proposition \ref{prop1} may look unsatisfactory since the links between the steady Mach stem bifurcation problem and the evolutionary 
stability analysis encoded in the function $\Delta$ are still hidden. In the following Section, we explain in a more explicit way why 
determining the velocity $c_\star$ in \cite{MR2} corresponds to finding a nonzero solution $(\chi,\dot{U}) \in \C \times E^s (0,1)$ to the 
linear system \eqref{linearRH} in which the parameters $z,\eta,\overline{u}$ are respectively set equal to $0,1,-c_\star$. The existence 
of such a nontrivial solution is possible if and only if the Lopatinskii determinant vanishes, which gives a more satisfactory justification 
for Proposition \ref{prop1} than our preliminary `brute force' calculations. We also clarify below the links between the quantities in 
\eqref{defbetac} and various expressions that arise in the shock wave stability problem.

\section{Mach stems bifurcating from a steady planar shock}
\label{sect:MachStem}

In all what follows, we keep the notation $U$ for the four component vector $(\tau,{\bf u},s)$ and we keep the notation ${\bf u}=(u,v)$ 
for the velocity. We review the Mach stem bifurcation problem considered in \cite{MR2} and explain its link with the shock wave stability 
problem (and more precisely its formulation \eqref{RHequiv'}).

The problem considered in \cite{MR2} is the following: under which conditions on the shock \eqref{shock} is it possible to construct a family 
$(U_0,U_1,U_2,U_3,\Theta,\Phi,\Psi)(\eps)$, with $\eps \in [0,\eps_0]$ for some $\eps_0>0$, such that, for all $\eps \in (0,\eps_0]$, the 
wave pattern depicted in Figure \ref{fig:machstem} defines a steady Mach stem, and\footnote{Up to a rotation, the shock $\bfS_2(\eps)$ 
is always chosen as the $(0x_1)$ axis, which yields the convention for the angles $\Phi_0$ and $\Psi_0$. The location of the contact 
discontinuity with respect to $\bfS_3(\eps)$ follows from the convention $\overline{u} \le 0$, the symmetric situation being obtained by 
changing $\overline{u}$ into its opposite and $x_1$ into $-x_1$.}
\begin{align}
& \lim_{\eps \to 0} \, U_0(\eps) \, = \, \underline{U}_0 \, , & 
& \lim_{\eps \to 0} \, U_1(\eps) \, = \, \lim_{\eps \to 0} \, U_2(\eps) \, = \, \lim_{\eps \to 0} \, U_3(\eps) \, = \, \underline{U}_1 \, , \label{limitMS} \\
& \lim_{\eps \to 0} \, \Theta(\eps) \, = \, \pi \, , & 
& \lim_{\eps \to 0} \, \Phi(\eps) \, = \, \Phi_0 \in (\pi,3\, \pi/2) \, ,\notag \\
& & & \lim_{\eps \to 0} \, \Psi(\eps) \, = \, \Psi_0 \in (\Phi_0,2\, \pi) \, .\notag
\end{align}
It is also required that the family depends smoothly, say at least in a ${\mathcal C}^1$ way, on $\eps$. By a Mach stem, we 
mean that Figure \ref{fig:machstem} represents a four wave interaction at the origin where, for all $\eps \in (0,\eps_0]$:
\begin{itemize}
 \item $\bfS_1(\eps)$ is a steady shock front where $U_0(\eps)$ is the state ahead of the shock and $U_2(\eps)$ the state behind,
 
 \item $\bfS_2(\eps)$ is a steady shock front where $U_0(\eps)$ is the state ahead of the shock and $U_1(\eps)$ the state behind,
 
 \item $\bfS_3(\eps)$ is a steady shock front where $U_1(\eps)$ is the state ahead of the shock and $U_3(\eps)$ the state behind,
 
 \item ${\bf CD}(\eps)$ is a steady contact discontinuity.
\end{itemize}
In that case, the wave pattern depicted in Figure \ref{fig:machstem} yields a weak solution to the steady Euler equations. The following 
\emph{causality} conditions are also imposed for all $\eps \in (0,\eps_0]$, see \cite[page 123]{MR2}:
\begin{itemize}
 \item The tangential velocity of the fluid along $\bfS_3(\eps)$ points away from the origin.
 
 \item The velocity on either side of ${\bf CD}(\eps)$ points away from the origin.
\end{itemize}

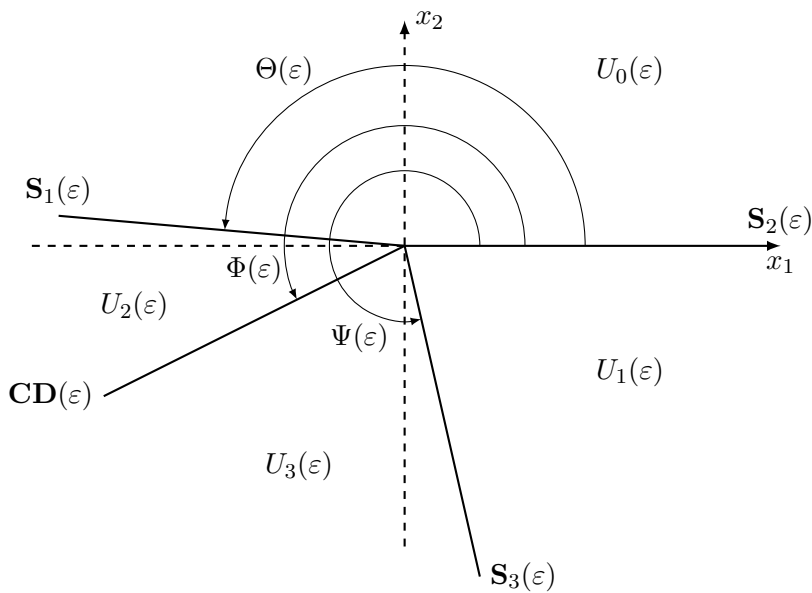
\begin{figure}[h!]
\begin{center}
\begin{tikzpicture}[scale=2,>=latex]
\pgfmathsetmacro{\theta}{ {4*atan(1)-atan(2/23)} }
\pgfmathsetmacro{\phi}{ {atan(1/2)+4*atan(1)} }
\pgfmathsetmacro{\psi}{ {8*atan(1)-atan(22/5)} }
\draw[thick,->] (0,0) -- (2.5,0) node[below] {$x_1$};
\node[above] at (2.5,0) {$\bfS_2(\eps)$};
\draw[thick,dashed,->] (0,-2)--(0,1.5) node[right] {$x_2$};
\draw[thick,dashed] (0,0) -- (-2.5,0);
\draw[thick] (0,0) -- (-2.3,0.2);
\node[above] at (-2.3,0.2) {$\bfS_1(\eps)$};
\node[above] at (1.5,1) {$U_0(\eps)$};
\node[above] at (1.5,-1) {$U_1(\eps)$};
\node[left] at (-1.5,-0.4) {$U_2(\eps)$};
\node[below] at (-0.7,-1.3) {$U_3(\eps)$};
\draw[thick] (0,0) -- (0.5,-2.2);
\node[right] at (0.5,-2.2) {$\bfS_3(\eps)$};
\draw[thick] (0,0) -- (-2,-1);
\node[left] at (-2,-1) {${\bf CD}(\eps)$};
\draw[->] (1.2,0) arc (0:\theta:1.2);
\node[above] at (-0.8,1) {$\Theta(\eps)$};
\draw[->] (0.8,0) arc (0:\phi:0.8);
\node[above] at (-1,-0.32) {$\Phi(\eps)$};
\draw[->] (0.5,0) arc (0:\psi:0.5);
\node[below] at (-0.3,-0.45) {$\Psi(\eps)$};
\end{tikzpicture}
\caption{The Mach stem configuration with a shallow angle.}
\label{fig:machstem}
\end{center}
\end{figure}

\noindent By the so-called triple shock Theorem \cite{HendersonMenikoff}, $\bfS_1(\eps)$ should be the shock with the largest 
amplitude while $\bfS_2(\eps)$ will have a slightly smaller amplitude. In other words, we should have:
$$
p(U_0(\eps)) \, < \, p(U_1(\eps)) \, < \, p(U_2(\eps)) \, = \, p(U_3(\eps)) \, .
$$

There is of course much freedom in the parametrization of the family of Mach stems (and we shall see below that even with the 
choice of one parametrization, there can be more than one family of Mach stems that satisfy the asymptotic behavior \eqref{limitMS}). 
As will follow from the proof of Theorem \ref{thm1} below, it turns out that a convenient  way to parametrize the family of Mach stems 
amounts to choosing $\pi -\Theta$ as the `bifurcation parameter'. In other words, we ask from now on the angle $\Theta$ to be given 
by
\begin{equation}
\label{Theta}
\Theta(\eps) \, := \, \pi -\eps\, .
\end{equation}
Why $\eps$ should be positive, and therefore $\Theta$ in \eqref{Theta} less than $\pi$, will be made clear in the proof of Theorem 
\ref{thm1} below (see also \cite{MR2} for similar considerations).

The main conclusion in \cite{MR2} is that such a `shock to Mach stem' bifurcation occurs if and only if the inequalities 
\eqref{weakstab} are satisfied and the tangential velocity $\overline{u}$ in \eqref{RH1} equals $-c_\star$, with $c_\star$ 
defined by \eqref{eqPhi}, \eqref{defbetac}. The proof of the `if' part in \cite{MR2} is skipped. Let us observe that because 
of our result in Proposition \ref{prop1}, the requirement on $\overline{u}$ in \cite{MR2} can be rewritten as $\overline{u}=-V$ 
where $V>0$ is characterized by \eqref{caracV}. Our main conclusion is of course in agreement with \cite{MR2} and is 
summarized as follows (we recall the convention $\overline{u} \le 0$).

\begin{theorem}
\label{thm1}
Assume that there exists a family of Mach stems depending smoothly on $\eps \in [0,\eps_0]$ where $\eps$ is given by \eqref{Theta}, 
and having the asymptotic behavior \eqref{limitMS} where the steady shock \eqref{shock} satisfies \eqref{RH}. Then there exists a 
solution $\dot{U} \in E^s(0,1)$ to the linear system \eqref{linearRH} in which one specifies $z=0$, $\eta=1$ and $\chi=1$. In particular, 
there holds:
\begin{equation*}
\dfrac{1}{1+\Gamma_1} \le M_1^2 \, \left( \dfrac{\tau_0}{\tau_1}-1 \right) < \dfrac{1+M_1}{\Gamma_1} \, ,
\end{equation*}
and $\overline{u}=-V$, where the velocity $V$ is characterized by \eqref{caracV}. Equivalently, the step shock \eqref{shock}, \eqref{RH} 
with tangential velocity $\overline{u}$ can bifurcate into a Mach stem with shallow angle only if there holds $\Delta (\overline{u},0,1)=0$.
\bigskip

Conversely, if the steady shock \eqref{shock} satisfies \eqref{RH} together with the strict inequalities \eqref{weakstab}, and if furthermore 
$\overline{u}=-V$ where the velocity $V$ is characterized by \eqref{caracV}, then there exists a one parameter family of Mach stems 
satisfying \eqref{limitMS}, \eqref{Theta}, and the state $U_0(\eps)$ ahead of the Mach stem can be chosen to have the particular form
\begin{equation}
\label{formU0eps}
\forall \, \eps \in [0,\eps_0] \, ,\quad U_0(\eps) \, = \, (\tau_0,u(\eps),v_0,s_0) \, .
\end{equation}
\end{theorem}

Our proof slightly differs from \cite{MR2} and intends to clarify the link between the Mach stem bifurcation analysis and the relation 
\eqref{RHequiv'} which deals with the normal mode analysis for the shock wave stability problem. Let us recall for future use that 
the validity of \eqref{RHequiv'} is equivalent to the existence of a nonzero pair $(\chi,\dot{U}) \in \C \times E^s(z,\eta)$ satisfying 
\eqref{linearRH}, which is also equivalent to $\Delta (\overline{u},z,\eta)=0$. In the case where $z$ is real, this corresponds to a 
weak stability property for the shock wave \eqref{shock}. In the proof of Theorem \ref{thm1} below, we clarify some of the computations 
and arguments in \cite{MR2} and, above all, we provide a complete proof of the `if' part in Theorem \ref{thm1} without any additional 
assumption on the pressure law than those made in Section \ref{normal}. The assumptions on the pressure law were not complete 
in \cite{MR2}.

\begin{proof}[Proof of Theorem \ref{thm1}]
We first prove the `necessity' part of Theorem \ref{thm1} and assume that a smooth family of Mach stems with shallow angle is given. 
We first introduce some notation and write\footnote{In \cite{MR2}, the state $U_0'$ is \emph{assumed} to have only one nonzero 
component, which corresponds to the tangential velocity. This assumption will be justified when we construct the family of Mach 
stems but it is actually not needed at this point of the argument.}:
\begin{equation*}
U_0(\eps) \, = \, \underline{U}_0 +\eps \, U_0' +o(\eps) \, ,\qquad U_j(\eps) \, = \, \underline{U}_1 +\eps \, U_j' +o(\eps) \, , \, 
\quad j \, = \, 1,2,3 \, ,
\end{equation*}
where $\underline{U}_0,\underline{U}_1$ are the two states of the step shock in \eqref{shock}. Following \cite[page 130]{MR2}, 
we write the stationary Rankine-Hugoniot conditions for each of the four discontinuities $\bfS_{1,2,3}(\eps)$, ${\bf CD}(\eps)$ 
and expand at the first order in $\eps$. This yields
\begin{align}
& {\rm d}f_2(\underline{U}_0) \, U_0' -{\rm d}f_2(\underline{U}_1) \, U_1' \, = \, 0 \, ,\quad 
f_1(\underline{U}_0) -f_1(\underline{U}_1) +{\rm d}f_2(\underline{U}_0) \, U_0' -{\rm d}f_2(\underline{U}_1) \, U_2' \, = \, 0 \, ,\label{saut1} \\
&\Big( -\sin \Psi_0 \, {\rm d}f_1(\underline{U}_1) +\cos \Psi_0 \, {\rm d}f_2(\underline{U}_1) \Big) \, (U_1' -U_3') \, = \, 0 \, ,\label{saut2} \\
&\Big( -\sin \Phi_0 \, {\rm d}f_1(\underline{U}_1) +\cos \Phi_0 \, {\rm d}f_2(\underline{U}_1) \Big) \, (U_3' -U_2') \, = \, 0 \, .\label{saut3}
\end{align}
We can also pass to the limit in the zero normal velocity constraint for the contact discontinuity ${\bf CD}(\eps)$ and in Lax shock 
inequalities for $\bfS_3(\eps)$, which yields (recall ${\bf u}_1=(\overline{u},v_1)$):
\begin{equation}
\label{egalPhiPsi}
-\overline{u} \, \sin \Phi_0 +v_1 \, \cos \Phi_0 \, = \, 0 \, ,\quad 
-\overline{u} \, \sin \Psi_0 +v_1 \, \cos \Psi_0 \, = \, -c_1 \, .
\end{equation}
Our goal now is to analyze the relations \eqref{saut1}, \eqref{saut2}, \eqref{saut3} and \eqref{egalPhiPsi}. Observe in particular 
that we shall never use the precise expression of $U_0'$, meaning that the assumption made in \cite{MR2} that $U_0'$ only has 
one nonzero component is useless.

We can first equivalently rewrite \eqref{saut1} as
\begin{subequations}
\begin{align}
& {\rm d}f_2(\underline{U}_0) \, U_0' -{\rm d}f_2(\underline{U}_1) \, U_1' \, = \, 0 \, ,\label{saut4a} \\
& U_1' -U_2' \, = \, {\rm d}f_2(\underline{U}_1)^{-1} \, \big( f_1(\underline{U}_1) -f_1(\underline{U}_0) \big) \, .\label{saut4b}
\end{align}
\end{subequations}
Equation \eqref{saut4a} will ultimately determine $U_0'$ once we know $U_1'$. Observe now that $U_0'$ does not appear in 
\eqref{saut4b}, \eqref{saut2} and \eqref{saut3} so we may forget about $U_0'$ from now on and focus on the vectors $U_j'$, 
$j=1,2,3$. We are going to show below that the relations \eqref{saut2}, \eqref{saut3} imply $U_1' -U_2' \in E^s(0,1)$ where the 
stable subspace $E^s(0,1)$ has been defined in Section \ref{normal}, see \eqref{eigenspaces01}. Therefore \eqref{saut4b} will 
give
$$
{\rm d}f_2(\underline{U}_1)^{-1} \, \big( f_1(\underline{U}_1) -f_1(\underline{U}_0) \big) \in E^s(0,1) \, ,
$$
which is exactly \eqref{RHequiv'} with $z=0$ and $\eta=1$. In other words, we shall have $\Delta (\overline{u},0,1)=0$, which 
equivalently means that \eqref{weakstab} holds and $\overline{u} =-V$ with our convention for the sign of $\overline{u}$. Let us 
make all these arguments precise.

We are first going to make the relations \eqref{saut2}, \eqref{saut3} more explicit. Recalling that $U$ denotes the vector 
$(\tau,{\bf u},s)$, we compute
\begin{equation*}
{\rm d}f_1(U) \, = \, P(U)^{-1} \, \begin{bmatrix}
u & -\tau & 0 & 0 \\
\dfrac{-c^2}{\tau} & u & 0 & \Gamma \, T \\
0 & 0 & u & 0 \\
0 & 0 & 0 & u \end{bmatrix} \, ,\quad 
{\rm d}f_2(U) \, = \, P(U)^{-1} \, \begin{bmatrix}
v & 0 & -\tau & 0 \\
0 & v & 0 & 0 \\
\dfrac{-c^2}{\tau} & 0 & v & \Gamma \, T \\
0 & 0 & 0 & v \end{bmatrix} \, ,
\end{equation*}
with
\begin{equation}
\label{defPU}
P(U) \, := \, \begin{bmatrix}
-\tau^2 & 0 & 0 & 0 \\
-\tau \, u & \tau & 0 & 0 \\
-\tau \, v & 0 & \tau & 0 \\
\dfrac{\tau^2}{T} \, \left( \rho \, \left( \dfrac{|{\bf u}|^2}{2} -e \right) -p \right) & \dfrac{-\tau \, u}{T} 
& \dfrac{-\tau \, v}{T} & \dfrac{\tau}{T} \end{bmatrix} \, .
\end{equation}
Multiplying \eqref{saut3} by $P(\underline{U}_1)$ on the left and using the first equation in \eqref{egalPhiPsi}, \eqref{saut3} is 
seen to be equivalent to
\begin{equation*}
U_3' -U_2' \in \text{\rm Span } \left\{ \, 
\begin{bmatrix}
0 \\
\cos \Phi_0 \\
\sin \Phi_0 \\
0 \end{bmatrix} \, , \, \begin{bmatrix}
\Gamma_1 \, T_1 \, \tau_1 \\
0 \\
0 \\
c_1^2 \end{bmatrix} \, \right\} \, ,\quad \text{\rm with } \quad \dfrac{\cos \Phi_0}{\sin \Phi_0} \, = \, \dfrac{\overline{u}}{v_1} \, .
\end{equation*}
Recalling the decomposition \eqref{eigenspaces01} of the vector space $E_0(0,1)$, it thus turns out that \eqref{saut3} equivalently reads 
$U_3' -U_2' \in E_0(0,1)$.

Let us now turn to \eqref{saut2}. Multiplying \eqref{saut2} by $P(\underline{U}_1)$ on the left and using the second equation in 
\eqref{egalPhiPsi}, \eqref{saut2} is seen to be equivalent to
\begin{equation}
\label{conditionsaut}
U_1' -U_3' \in \text{\rm Span } \begin{bmatrix}
\tau_1 \\
c_1 \, \sin \Psi_0 \\
-c_1 \, \cos \Psi_0 \\
0 \end{bmatrix} \, ,
\end{equation}
where the angle $\Psi_0$ satisfies the second equation in \eqref{egalPhiPsi}. We wish to show that \eqref{conditionsaut} equivalently 
reads $U_1' -U_3' \in E_-(0,1)$ where the vector space $E_-(0,1)$ is given in \eqref{eigenspaces01}.  We thus need to show that the 
two vectors
$$
\begin{bmatrix}
\tau_1 \, (\overline{u}+v_1 \, \omega_-(0,1)) \\
c_1^2 \\
c_1^2 \, \omega_-(0,1) \\
0 \end{bmatrix} \, ,\quad \begin{bmatrix}
\tau_1 \\
c_1 \, \sin \Psi_0 \\
-c_1 \, \cos \Psi_0 \\
0 \end{bmatrix} \, ,
$$
are collinear, or in other words, we need to show the relations
\begin{equation}
\label{sincosPsi_0}
\sin \Psi_0 \, = \, \dfrac{c_1}{\overline{u} +v_1 \, \omega_-(0,1)} \, ,\quad 
\cos \Psi_0 \, = \, -\dfrac{c_1 \, \omega_-(0,1)}{\overline{u} +v_1 \, \omega_-(0,1)} \, .
\end{equation}
This requires further knowledge on the angle $\Psi_0$ than just \eqref{egalPhiPsi}. We are therefore going to determine the angle 
$\Psi_0$ completely.

By the causality conditions for the family of Mach stems, the velocity along the contact discontinuity ${\bf CD}(\eps)$ must point away 
from the origin. Passing to the limit $\eps \to 0$, this means that the velocity ${\bf u}_1$ is a positive multiple of the unit vector 
$(\cos \Phi_0,\sin \Phi_0)$. In other words, we have
$$
\cos \Phi_0 \, = \, \dfrac{\overline{u}}{\sqrt{\overline{u}^2+v_1^2}} \, ,\quad \quad 
\sin \Phi_0 \, = \, \dfrac{v_1}{\sqrt{\overline{u}^2+v_1^2}} \, .
$$
Let us now recall that the angle $\Psi_0$ should satisfy $\Psi_0 \in (\Phi_0,2\, \pi)$ so we must have
$$
\cos \Psi_0 \in (\cos \Phi_0,1) \, ,\quad \text{\rm and } \quad \sin \Psi_0 < 0 \, .
$$
Let us now go back to \eqref{egalPhiPsi}. Since $\sin \Psi_0$ is negative, $\cos \Psi_0$ must satisfy
\begin{equation}
\label{eqbeta}
-(c_1+v_1 \, \cos \Psi_0) \, = \, \overline{u} \, \sqrt{1-\cos^2 \Psi_0} \, < \, 0 \, .
\end{equation}
Moreover, the existence of $\Psi_0$ implies, by the Cauchy-Schwarz inequality, the condition\footnote{This means that the velocity field 
${\bf u}_1$ behind the shock in \eqref{shock} is supersonic though the normal velocity $v_1$ is subsonic by Lax shock inequalities.}
\begin{equation}
\label{hyperbolic}
\overline{u}^2 \ge c_1^2 -v_1^2 \, .
\end{equation}
Solving \eqref{eqbeta} for $\cos \Psi_0$, we obtain the two possible values
$$
\cos \Psi_0 \, = \, \dfrac{-c_1 \, v_1 \pm \overline{u} \, \sqrt{\overline{u}^2 +v_1^2 -c_1^2}}{\overline{u}^2 +v_1^2} \, ,
$$
both values belonging to the interval $(\cos \Phi_0,1)$. Since both possible values correspond to a positive value of 
$c_1+v_1 \, \cos \Psi_0$, we obtain
$$
\sin \Psi_0 \, = \, -\sqrt{1-\cos^2 \Psi_0} \, = \, \dfrac{c_1+v_1 \, \cos \Psi_0}{\overline{u}} \, = \, 
\dfrac{c_1 \, \overline{u} \pm v_1 \, \sqrt{\overline{u}^2 +v_1^2 -c_1^2}}{\overline{u}^2 +v_1^2} \, ,
$$
where the choice between $\pm$ is the same in both expressions of $\cos \Psi_0$ and $\sin \Psi_0$. At this stage, it is unfortunately 
still not possible to select the appropriate determination of $\Psi_0$ between the two possible choices. However, the causality conditions 
for the family of Mach stems imply
\begin{equation*}
\forall \, \eps \in (0,\eps_0] \, ,\quad {\bf u}_1(\eps) \cdot \big( \cos \Psi(\eps),\sin \Psi(\eps) \big) \ge 0 \, ,
\end{equation*}
so passing to the limit, we must have
\begin{equation*}
\overline{u} \, \cos \Psi_0 +v_1 \, \sin \Psi_0 \ge 0 \, .
\end{equation*}
Therefore the appropriate determination of $\Psi_0$ corresponds to
\begin{equation}
\label{determinangle}
\cos \Psi_0 \, = \, \dfrac{-c_1 \, v_1 +\overline{u} \, \sqrt{\overline{u}^2 +v_1^2 -c_1^2}}{\overline{u}^2 +v_1^2} \, ,\quad 
\sin \Psi_0 \, = \, \dfrac{c_1 \, \overline{u} +v_1 \, \sqrt{\overline{u}^2 +v_1^2 -c_1^2}}{\overline{u}^2 +v_1^2} \, ,
\end{equation}
which amounts to choosing the smallest possible value of $\cos \Psi_0$ among the two possible ones, see the geometric 
interpretation in Figure \ref{fig:angle}.

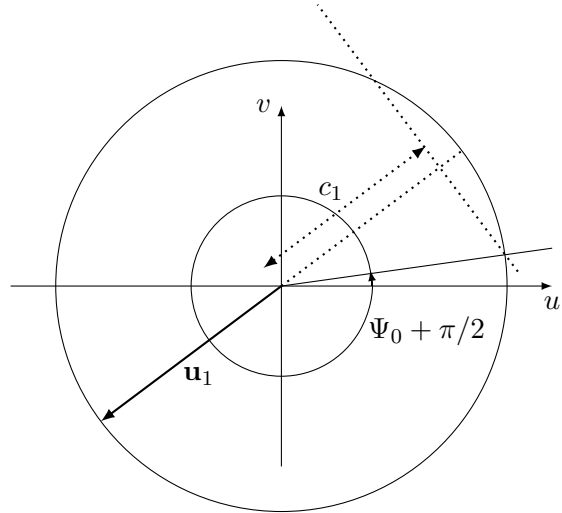
\begin{figure}[h!]
\begin{center}
\begin{tikzpicture}[scale=1.2,>=latex]
\draw[thin,->] (-3,0) -- (3,0);
\node[below] at (3,0) {$u$};
\draw[thin,->] (0,-2) -- (0,2);
\node[left] at (0,2) {$v$};
\draw[thick,->] (0,0) -- (-2,-1.5);
\draw[thin] (2.5,0) arc (0:360:2.5);
\node[right] at (-1.2,-1) {${\bf u}_1$};
\draw[thick,dotted] (0,0) -- (2,1.5);
\draw[thick,dotted,<->] (-0.2,0.2) -- (1.6,1.55);
\node[left] at (0.8,1.05) {$c_1$};
\draw[thick,dotted] (2.62,0.16) -- (0.4,3.12);
\draw[->] (1,0) arc (0:370:1);
\draw (0,0) -- (3,0.42);
\node[right] at (0.85,-0.5) {$\Psi_0+\pi/2$};
\end{tikzpicture}
\caption{Determining the angle $\Psi_0$. One determines the two possible angles $\Psi_0+\pi/2$ by the scalar product condition 
in \eqref{egalPhiPsi} and picks up the smallest value. When $c_1$ equals $|{\bf u}_1|$, the two possible values of $\Psi_0+\pi/2$ 
become equal.}
\label{fig:angle}
\end{center}
\end{figure}

Recalling now the expression \eqref{eigenmode01} of $\omega_-(0,1)$ and the expressions \eqref{determinangle} for the sine and 
cosine of $\Psi_0$, it is a simple exercise to verify that the relations \eqref{sincosPsi_0} hold. Going then back to \eqref{conditionsaut}, 
we have obtained $U_1' -U_3' \in E_-(0,1)$, and combining with $U_3' -U_2' \in E_0(0,1)$, we have shown $U_1' -U_2' \in E^s(0,1)$.

Summarizing what we have done so far, we have shown that the existence of a family of Mach stems satisfying \eqref{limitMS}, \eqref{Theta} 
and the causality conditions implies that the tangential velocity $\overline{u}$ satisfies \eqref{hyperbolic} and that there exists a vector $U_1' 
-U_2' \in E^s(0,1)$ such that \eqref{saut4b} holds. This means that the relation \eqref{RHequiv'} holds for the frequencies $(z,\eta)=(0,1)$ 
hence the step shock \eqref{shock} satisfies $\Delta (\overline{u},0,1)=0$. Using Lemma \ref{determinant-velocity}, we thus have 
$\Delta (0,\overline{u},1)=0$ and we shall now use the conclusions of the stability analysis for shock waves with zero tangential velocity.

There are two cases. Either \eqref{hyperbolic} holds with an equality sign, and in that case, one gets\footnote{This limit case corresponds 
to a Lopatinskii determinant that vanishes at a glancing frequency. At such frequencies, the eigenmode $\omega_-$ is not locally a smooth 
function of $(z,\eta)$.}:
\begin{equation*}
\overline{u} \, = \, -\sqrt{c_1^2-v_1^2} \, ,\quad \Delta (0,\overline{u},1) \, = \, 0 \, .
\end{equation*}
This situation occurs only if (we refer again to \cite[chapter 15]{BS})
\begin{equation*}
\dfrac{1}{1+\Gamma_1} \, = \, M_1^2 \, \left( \dfrac{\tau_0}{\tau_1}-1 \right) \, ,
\end{equation*}
which is the limit case \eqref{weakstablim} for \eqref{weakstab}. Or \eqref{hyperbolic} holds with a strict inequality, which means that the 
Lopatinskii determinant $\Delta (0,\cdot,1)$ vanishes in the hyperbolic region. In that case, the strict inequalities \eqref{weakstab} hold and 
we necessarily have $\overline{u}=-V$, see again \cite[chapter 15]{BS} for further details. This completes the proof of the `necessity' part of 
Theorem \ref{thm1}.
\bigskip

We are now going to give a complete construction of a family of Mach stems satisfying \eqref{limitMS}, \eqref{Theta} and the causality 
conditions. We are even going to show that the state $U_0(\eps)$ ahead of the Mach stem can be chosen to have the particular form 
\eqref{formU0eps} for some smooth function $u$ verifying $u(0)=\overline{u}$. This justifies the calculations made in \cite{MR2}. We 
assume from now on that the inequalities \eqref{weakstab} hold and that the tangential velocity $\overline{u}$ equals $-V$ with $V$ 
characterized by \eqref{caracV}. Where all these assumptions come into play will be made precise below. The construction of the family 
of Mach stems splits in several steps. Let us recall some important facts. The shocks ${\bf S}_1(\eps)$ and ${\bf S}_2(\eps)$ will have a 
`large' amplitude, meaning that their amplitude will not tend to zero with $\eps$. At the opposite, both discontinuities ${\bf S}_3(\eps)$ and 
${\bf CD}(\eps)$ will have small amplitude, which explains why we deal with those two sets of discontinuities by using separate arguments.

We now wish to construct a one parameter family of Mach stems satisfying \eqref{limitMS}, \eqref{Theta}. We shall see later 
on which degrees of freedom on the state $U_0$ can be spared. The family of Mach stems will be constructed by repeatedly 
applying the implicit function Theorem (and at the very end of the argument, a somehow degenerate version of that Theorem). 
Some arguments in the construction are inspired from \cite[chapter 4]{serre2}. We split the whole construction in several steps.
\bigskip

$\bullet$ \underline{Step 1. The large amplitude shocks.}

Let us first set some notation. By applying the implicit function Theorem\footnote{For simplicity, we assume that 
the specific internal energy $e$ is a ${\mathcal C}^\infty$ function of $(\tau,s)$ so all functions involved in the 
calculations below have ${\mathcal C}^\infty$ regularity.}, we know that for all $U$ close to $\underline{U}_0$ 
and for all $\eps$ close to $0$, there exist some uniquely determined states $\U_1(U)$ and $\U_2(\eps,U)$ that 
are close to $\underline{U}_1$ and satisfy the Rankine-Hugoniot relations
\begin{align}
&f_2(\U_1(U)) -f_2(U) \, = \, 0 \, ,\quad 
\big( f_1(\U_2(\eps,U)) -f_1(U) \big) \sin \eps +\big( f_2(\U_2(\eps,U)) -f_2(U) \big) \cos \eps \, = \, 0 \, ,\label{defU12} \\
&\U_1(\underline{U}_0) \, = \, \U_2(0,\underline{U}_0) \, = \, \underline{U}_1 \, .\notag
\end{align}
By uniqueness in the implicit function Theorem, we even have the relation
\begin{equation*}
\U_2(0,U) \, = \, \U_1(U) \, ,
\end{equation*}
for all $U$ sufficiently close to $\underline{U}_0$. The relations \eqref{defU12} correspond to admissible discontinuities 
for the stationary Euler equations across the lines $\{ x_2=0 \}$ and $\{ \sin \eps \, x_1 +\cos \eps \, x_2 =0 \}$, as depicted 
in Figure \ref{fig:machstem}. At the very end of the analysis, we shall specify the state $U$ close to $\underline{U}_0$ by 
choosing $U=U_0(\eps)$ for some appropriate smooth function $U_0$. Then the discontinuity:
$$
\begin{cases}
U_0(\eps) \, , & x_2>0 \, ,\\
\U_1(U_0(\eps)) \, , & x_2<0 \, ,
\end{cases}
$$
will satisfy the Rankine-Hugoniot jump conditions for \eqref{euler}, and since $(U_0(\eps),\U_1(U_0(\eps)))$ will be close 
to $(\underline{U}_0,\underline{U}_1)$, this discontinuity will necessarily be a shock wave since it will have a nonzero mass 
flux across $\{ x_2=0 \}$ and it will satisfy Lax shock inequalities. In the same way, the discontinuity:
$$
\begin{cases}
U_0(\eps) \, , & \sin \eps \, x_1 +\cos \eps \, x_2>0 \, ,\\
\U_2(\eps,U_0(\eps)) \, , & \sin \eps \, x_1 +\cos \eps \, x_2<0 \, ,
\end{cases}
$$
will be a shock wave for \eqref{euler}. For both discontinuities, $U_0(\eps)$ will be the upstream state. Up to the choice of 
the function $(\eps \mapsto U_0(\eps))$, which will be specified later on, this means that we have already constructed the 
two large amplitude shocks of the Mach stem.
\bigskip

$\bullet$ \underline{Step 2. The small amplitude shock.}

Let us now move on to the small amplitude shock $\bfS_3(\eps)$ in the Mach stem configuration. The arguments follow 
\cite[chapter 4]{serre2} here. Given the state $\underline{U}_1$ with $\overline{u}=-V$ (and therefore $\overline{u}^2 +v_1^2>c_1^2$), 
we define the angle $\Psi_0 \in (\pi,2\, \pi)$ by \eqref{determinangle}. In particular, $\Psi_0$ satisfies the second equation 
in \eqref{egalPhiPsi} and there holds
\begin{equation*}
\overline{u} \, \cos \Psi_0 +v_1 \, \sin \Psi_0 \, = \, \sqrt{\overline{u}^2 +v_1^2 -c_1^2} \, \neq \, 0 \, ,
\end{equation*}
use \eqref{determinangle}. For all pair $(U_1,U_3)$ sufficiently close to $(\underline{U}_1,\underline{U}_1)$, we define the matrices
\begin{equation}
\label{defA1A2}
A_j(U_1,U_3) \, := \, \int_0^1 {\rm d}f_j \big( t\, U_1+(1-t)\, U_3 \big) \, {\rm d}t \, ,\quad j \, = \, 1,2 \, .
\end{equation}
Given an angle $\Psi$, the Rankine-Hugoniot conditions
\begin{equation*}
-\big( f_1(U_3) -f_1(U_1) \big) \sin \Psi +\big( f_2(U_3) -f_2(U_1) \big) \cos \Psi \, = \, 0 \, ,
\end{equation*}
with $U_1 \neq U_3$, equivalently read
\begin{align}
&\det \Big( -\sin \Psi \, A_1(U_1,U_3) +\cos \Psi \, A_2(U_1,U_3) \Big) \, = \, 0 \, ,\label{defPsiU13} \\
&U_3 -U_1 \in \text{\rm Ker } \Big( -\sin \Psi \, A_1(U_1,U_3) +\cos \Psi \, A_2(U_1,U_3) \Big) \, .\notag
\end{align}

Let us recall the expression:
\begin{multline*}
\det \Big( -\sin \Psi \, A_1(U,U) +\cos \Psi \, A_2(U,U) \Big) \, = \, \det P(U)^{-1} \, \times \, \big( -u \, \sin \Psi +v \, \cos \Psi -c \big) \\
\times \, \big( -u \, \sin \Psi +v \, \cos \Psi \big)^2 \, \times \, \big( -u \, \sin \Psi +v \, \cos \Psi +c \big) \, ,
\end{multline*}
where the invertible matrix $P(U)$ is given in \eqref{defPU}. We can therefore apply the implicit function Theorem and determine 
a unique angle $\Psi (U_1,U_3)$ that satisfies the eikonal equation \eqref{defPsiU13} for any pair $(U_1,U_3)$ sufficiently close 
to $(\underline{U}_1,\underline{U}_1)$, together with $\Psi (\underline{U}_1,\underline{U}_1) =\Psi_0$. In the case $U_1=U_3=U$, 
the angle $\Psi(U,U)$ is obtained by solving
\begin{equation*}
-u \, \sin \Psi +v \, \cos \Psi +c \, = \, 0 \, ,
\end{equation*}
with $\Psi(U,U)$ close to $\Psi_0$ ($u$ is close to $\overline{u}$, $v$ is close to $v_1$ and $c$ is close to $c_1$). This means that 
the angle $\Psi(U,U)$ is determined by the relations:
\begin{equation}
\label{anglequelconque}
\cos \Psi(U,U) \, = \, \dfrac{-c \, v +u \, \sqrt{u^2 +v^2 -c^2}}{u^2 +v^2} \, ,\quad 
\sin \Psi(U,U) \, = \, \dfrac{c \, u +v \, \sqrt{u^2 +v^2 -c^2}}{u^2 +v^2} \, .
\end{equation}

For the angle $\Psi =\Psi(U_1,U_3)$, the matrix $-\sin \Psi \, A_1(U_1,U_3) +\cos \Psi \, A_2(U_1,U_3)$ has a one-dimensional kernel that 
is spanned by a vector $R(U_1,U_3)$ which we can choose to satisfy the normalization condition
\begin{equation}
\label{conventionR}
R(U,U) \, = \, {\mathcal R}(U) \, := \, \begin{bmatrix}
\tau \\
c \, \sin \Psi(U,U) \\
-c \, \cos \Psi(U,U) \\
0 \end{bmatrix} \, ,
\end{equation}
where $\tau$ denotes the specific volume for the state $U$ and $c$ the sound speed associated with $U$. Then for any real number 
$\lambda$ close to $0$ and any state $U_1$ close to $\underline{U}_1$, the implicit function Theorem shows that there exists a uniquely 
determined $\U_3 (\lambda,U_1)$ satisfying the Rankine-Hugoniot jump conditions
\begin{equation}
\label{defU3}
\U_3(\lambda,U_1) -U_1 \, = \, \lambda \, R \big( U_1,\U_3(\lambda,U_1) \big) \, .
\end{equation}
In particular, the function $\U_3$ in \eqref{defU3} satisfies:
\begin{equation*}
\U_3(0,U_1) =U_1 \, ,\quad {\rm d}_{U_1} \U_3(0,U_1) =I \, ,\quad \partial_\lambda \U_3(0,U_1) ={\mathcal R}(U_1) \, ,
\end{equation*}
for all $U_1$ close to $\underline{U}_1$ (and the vector ${\mathcal R}(U_1)$ is given in \eqref{conventionR}).

At this point, given any $U_1$ sufficiently close to $\underline{U}_1$ and any $\lambda$ close to $0$, we have constructed a weak 
solution to \eqref{euler} of the form
$$
\begin{cases}
U_1 \, , & -\sin \Psi(U_1,\U_3(\lambda,U_1)) \, x_1 +\cos \Psi(U_1,\U_3(\lambda,U_1)) \, x_2 > 0 \, ,\\
\U_3(\lambda,U_1) \, , & -\sin \Psi(U_1,\U_3(\lambda,U_1)) \, x_1 +\cos \Psi(U_1,\U_3(\lambda,U_1)) \, x_2 < 0 \, .
\end{cases}
$$
It is not clear yet whether this discontinuity is a contact discontinuity or a shock wave (at this stage, it could even be a nonadmissible 
shock, that is it could be a noncharacteristic discontinuity that violates Lax shock inequalities).

Ultimately, the state $U_1$ will be chosen as $\U_1(U_0(\eps))$ and the state $U_3$ will therefore be determined as 
$\U_3(\lambda,\U_1(U_0(\eps)))$ for some appropriate amplitude $\lambda$ that will depend on the small parameter $\eps$. 
As a matter of fact, we are now going to choose the amplitude $\lambda$ appropriately in order to make the states $U_2$ and 
$U_3$ have equal pressures.
\bigskip

$\bullet$ \underline{Step 3. Adapting the small amplitude to get equal pressures for the states $2$ and $3$.}

Let us still keep the upstream state $U$ free, close to $\underline{U}_0$. Then we specify the state $U_1$ as $\U_1(U)$ 
and the state $U_2$ as $\U_2(\eps,U)$. We also specify the state $U_3$ as $\U_3(\lambda,\U_1(U))$, as shown in Figure 
\ref{fig:premachstem} below.

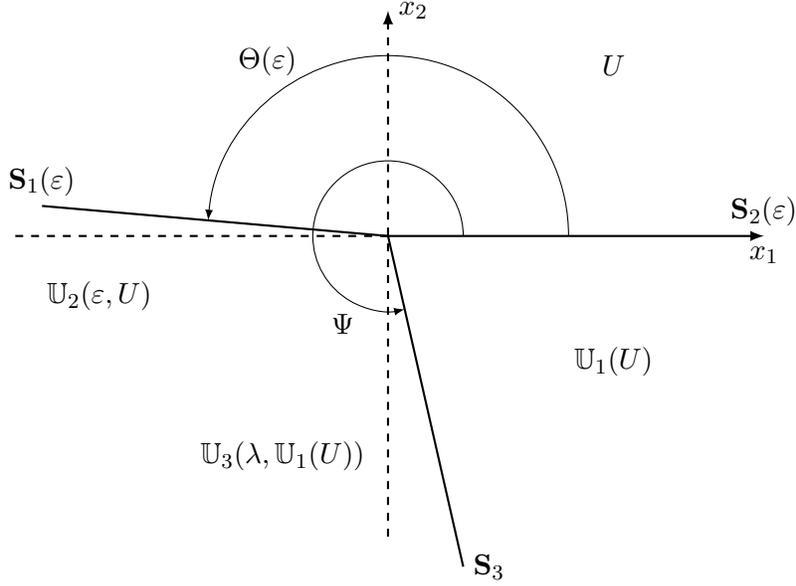
\begin{figure}[h!]
\begin{center}
\begin{tikzpicture}[scale=2,>=latex]
\pgfmathsetmacro{\theta}{ {4*atan(1)-atan(2/23)} }
\pgfmathsetmacro{\phi}{ {atan(1/2)+4*atan(1)} }
\pgfmathsetmacro{\psi}{ {8*atan(1)-atan(22/5)} }
\draw[thick,->] (0,0) -- (2.5,0) node[below] {$x_1$};
\node[above] at (2.5,0) {$\bfS_2(\eps)$};
\draw[thick,dashed,->] (0,-2)--(0,1.5) node[right] {$x_2$};
\draw[thick,dashed] (0,0) -- (-2.5,0);
\draw[thick] (0,0) -- (-2.3,0.2);
\node[above] at (-2.3,0.2) {$\bfS_1(\eps)$};
\node[above] at (1.5,1) {$U$};
\node[above] at (1.5,-1) {$\U_1(U)$};
\node[left] at (-1.5,-0.4) {$\U_2(\eps,U)$};
\node[below] at (-0.7,-1.3) {$\U_3(\lambda,\U_1(U))$};
\draw[thick] (0,0) -- (0.5,-2.2);
\node[right] at (0.5,-2.2) {$\bfS_3$};
\draw[->] (1.2,0) arc (0:\theta:1.2);
\node[above] at (-0.8,1) {$\Theta(\eps)$};
\draw[->] (0.5,0) arc (0:\psi:0.5);
\node[below] at (-0.3,-0.45) {$\Psi$};
\end{tikzpicture}
\caption{Towards the construction of the Mach stem configuration with a shallow angle. Here $\Psi$ is a short notation 
for $\Psi(\U_1(U),\U_3(\lambda,\U_1(U)))$. At this stage, there is no separation between the states $\U_2(\eps,U)$ and 
$\U_3(\lambda,\U_1(U))$.}
\label{fig:premachstem}
\end{center}
\end{figure}

In order to construct a Mach stem as depicted in Figure \ref{fig:machstem}, we need the states $\U_2(\eps,U)$ and $\U_3(\lambda,\U_1(U))$ 
to have equal pressures and collinear velocities, which will determine the angle $\Phi$ as their common argument\footnote{Both velocities 
will be small perturbations of ${\bf u}_1=(\overline{u},v_1)$ so collinearity here means proportionality with a positive scaling factor, which will 
uniquely determine the angle $\Phi$ close to $\Phi_0 \in (\pi,3\, \pi/2)$.}. Let us first adapt the (small) amplitude $\lambda$ in order to make 
the pressures of $\U_2(\eps,U)$ and $\U_3(\lambda,\U_1(U))$ equal. This relies again on the implicit function Theorem. Indeed, we consider 
the function
$$
(\lambda,\eps,U) \, \longmapsto \, p \big( \U_2(\eps,U) \big) - p \big( \U_3(\lambda,\U_1(U)) \big) \, ,
$$
where $p(U)$ denotes the pressure associated with a state $U$ (which, of course, only depends on the first and fourth coordinates of $U$, 
namely the specific volume and specific entropy). We wish to compute the partial derivative of the above function with respect to $\lambda$ at 
$(0,0,\underline{U}_0)$. We use $\partial_\lambda \U_3(0,\underline{U}_1)={\mathcal R}(\underline{U}_1)$, with ${\mathcal R}(\underline{U}_1)$ 
as in \eqref{conventionR}, and we find that the partial derivative with respect to $\lambda$ at $(0,0,\underline{U}_0)$ equals $c_1^2/\tau_1 \neq 0$. 
By applying the implicit function Theorem, we thus find that for all $\eps$ close to zero and for all $U$ close to $\underline{U}_0$, there exists a 
uniquely determined amplitude $\lambda (\eps,U)$ close to zero such that
\begin{equation*}
p \big( \U_2(\eps,U) \big) \, = \, p \big( \U_3(\lambda (\eps,U),\U_1(U)) \big) \, .
\end{equation*}
Because $\U_2(\eps,U)$ coincides with $\U_1(U)$ for $\eps=0$, and because there holds $\U_3(0,U_1)=U_1$, we find
$$
\lambda (0,U) \, = \, 0 \, ,
$$
for all $U$ close to $\underline{U}_0$. It is therefore convenient to rewrite $\lambda(\eps,U)$ as $\eps \, \lambda(\eps,U)$ for a new function 
$\lambda$, so that there holds
\begin{equation}
\label{relationlambda}
p \big( \U_2(\eps,U) \big) =p \big( \U_3(\eps \, \lambda (\eps,U),\U_1(U)) \big) \, ,
\end{equation}
for all $(\eps,U)$ close to $(0,\underline{U}_0)$.
\bigskip

$\bullet$ \underline{Step 4. Making the velocities collinear.}

At this point, for all `bifurcation parameter' $\eps$ small enough ($\eps$ has no prescribed sign so far), and for all state $U$ close to 
$\underline{U}_0$, we have constructed:
\begin{itemize}
 \item the discontinuity $\bfS_1(\eps)$ connecting $U$ to $\U_2(\eps,U)$ (the corresponding angle is $\Theta (\eps) =\pi-\eps$ as required 
 in \eqref{Theta}),
 \item the discontinuity $\bfS_2(\eps)$ connecting $U$ to $\U_1(U)$ (the corresponding angle equals zero, which is no loss of generality 
 up to rotating the axes),
 \item the discontinuity $\bfS_3(\eps)$ connecting $\U_1(U)$ to $\U_3(\eps \, \lambda (\eps,U),\U_1(U))$ (the corresponding angle equals 
 $\Psi (\U_1(U),\U_3(\eps \, \lambda (\eps,U),\U_1(U)))$). 
\end{itemize}
Whether the discontinuity $\bfS_3(\eps)$ is a shock wave is still undetermined at this point (it is not even known so far whether this 
discontinuity is admissible in the entropy sense). Let us also recall that the amplitude $\eps \, \lambda(\eps,U)$ of the discontinuity 
$\bfS_3(\eps)$ has been tuned so as to make the pressures of the states $\U_2(\eps,U)$ and $\U_3 (\eps \, \lambda (\eps,U),\U_1(U))$ 
equal, see \eqref{relationlambda}.

In what follows, we write ${\bf u}(U)$ to denote the velocity associated with any state $U$, that is the vector formed by the second and third 
coordinates of $U$. For any sufficiently small $\eps$ and any state $U$ close to $\underline{U}_0$, we define the quantity
\begin{equation}
\label{defdelta}
\delta (\eps,U) \, := \, \det \Big| {\bf u}(\U_2(\eps,U)) \, , \, {\bf u}(\U_3(\eps \, \lambda (\eps,U),\U_1(U))) \Big| \, .
\end{equation}
Our goal is to construct a state $U_0(\eps)$, possibly of the form \eqref{formU0eps}, such that $\delta (\eps,U_0(\eps))$ vanishes for all 
sufficiently small $\eps$. However, we are facing here a rather degenerate situation because the function $\delta$ satisfies
\begin{equation}
\label{propdelta}
\delta (0,U) \, = \, 0 \, ,
\end{equation}
for any state $U$ close to $\underline{U}_0$. Therefore there is no chance of proving, for instance, $\partial_u \delta (0,\underline{U}_0) 
\neq 0$ and conclude straightforwardly by the implicit function Theorem. As a matter of fact, the cancellation property \eqref{propdelta} 
shows that we can write
\begin{equation*}
\delta (\eps,U) \, = \, \eps \, \widetilde{\delta} (\eps,U) \, ,
\end{equation*}
and our only hope is to apply the implicit function Theorem to the rescaled function $\widetilde{\delta}$. There is however a price to pay, 
which is reminiscent of the analysis in \cite{MR2} and which confirms why $\overline{u}$ has to be fixed in some very specific way. Let 
us assume indeed that we are able to construct some \emph{smooth} $U_0(\eps)$ such that $U_0(0)=\underline{U}_0$ and $\delta 
(\eps,U_0(\eps))=0$ for all $\eps$ close to $0$. Then because the partial derivatives of $\delta$ with respect to $U$ at $(0,\underline{U}_0)$ 
vanish (use \eqref{propdelta}), we must necessarily have $\partial_\eps \delta (0,\underline{U}_0)=0$. It turns out, see below, that this 
`compatibility' condition on the function $\delta$ is an equivalent formulation of the constraint $\overline{u}=-V$ for the tangential velocity 
of $\underline{U}_0$ and $\underline{U}_1$. Let us therefore assume for now that we can prove the relation $\partial_\eps \delta 
(0,\underline{U}_0)=0$. In view of the factorization $\delta =\eps \, \widetilde{\delta}$, the relation $\partial_\eps \delta (0,\underline{U}_0)=0$ 
also reads
\begin{equation*}
\widetilde{\delta} (0,\underline{U}_0) \, = \, 0 \, .
\end{equation*}
In order to apply the implicit function Theorem to $\widetilde{\delta}$, we need to verify that some partial derivative of $\widetilde{\delta}$ 
with respect to one of the components of $U$ is nonzero. In what follows, we compute the partial derivative with respect to the second 
coordinate of $U$, which corresponds to the tangential velocity. We are going to show $\partial_u \widetilde{\delta} (0,\underline{U}_0) 
\neq 0$, or equivalently:
$$
\partial_{\eps,u}^2 \delta (0,\underline{U}_0) \, \neq \, 0 \, .
$$
This will yield the existence of a state $U_0(\eps)$ of the form \eqref{formU0eps} such that $\widetilde{\delta}(\eps,U_0(\eps))=0$ for all 
sufficiently small $\eps$. In particular, this will ultimately explain why in \cite{MR2} it was legitimate to expand only the tangential velocity 
of $U_0$ with respect to the small bifurcation parameter\footnote{To our knowledge, it is open, though likely, that the implicit function 
Theorem could also be applied with respect to the $v$ variable, which would yield another family of Mach stems bifurcating from the 
same reference step shock \eqref{shock}. In any case, there are necessarily infinitely many families of Mach stems that bifurcate from 
the reference shock \eqref{shock} since $\partial_{\eps,u}^2 \delta (0,\underline{U}_0) \neq 0$ implies that infinitely many linear combinations 
of partial derivatives such as $\mu \, \partial_{\eps,\tau}^2 \delta (0,\underline{U}_0) +\partial_{\eps,u}^2 \delta (0,\underline{U}_0)$ ($\mu$ 
small enough) are nonzero.}.

Summarizing the above arguments, there are two main points that remain to be clarified. We first need to prove the compatibility condition 
$\partial_\eps \delta (0,\underline{U}_0)=0$, and we also need to prove the invertibility condition $\partial_{\eps,u}^2 \delta (0,\underline{U}_0) 
\neq 0$.

\paragraph{The compatibility condition.} Let us compute the partial derivative of $\delta$ with respect to $\eps$. Specifying 
$U=\underline{U}_0$ in \eqref{defdelta}, we compute
\begin{equation}
\label{expressionderivee1}
\partial_\eps \delta (0,\underline{U}_0) \, = \, \det \big| {\bf u}_1 \, , \, {\bf u} (\dot{U}_3 -\dot{U}_2) \big| \, ,
\end{equation}
where we use the notation
\begin{subequations}
\label{defU2.U3.}
\begin{align}
\dot{U}_2 \, & \, := \, \partial_\eps \U_2(0,\underline{U}_0) 
\, = \, -{\rm d}f_2(\underline{U}_1)^{-1} \, \big( f_1(\underline{U}_1) -f_1(\underline{U}_0) \big) \, ,\label{expressionU2.} \\
\dot{U}_3 \, & \, := \, \dfrac{\partial}{\partial \eps} \U_3(\eps \, \lambda (\eps,\underline{U}_0),\underline{U}_1) \Big|_{\eps=0} 
\, = \, \lambda (0,\underline{U}_0) \, {\mathcal R}(\underline{U}_1) \, .\label{expressionU3.}
\end{align}
\end{subequations}
By differentiating the relation of equal pressures \eqref{relationlambda}, we have
\begin{equation*}
-\dfrac{c_1^2}{\tau_1} \, \tau \big( \dot{U}_2 -\dot{U}_3 \big) +\Gamma_1 \, T_1 \, s\big( \dot{U}_2 -\dot{U}_3 \big) 
=0 \, ,
\end{equation*}
where $\tau \big( \dot{U}_2 -\dot{U}_3 \big)$, resp. $s \big( \dot{U}_2 -\dot{U}_3 \big)$, denote the first, resp. fourth, coordinate of 
$\dot{U}_2 -\dot{U}_3$. This means that the vector $\dot{U}_2 -\dot{U}_3$ can be decomposed as:
$$
\dot{U}_2 -\dot{U}_3 \, = \, \dot{\mu} \, \begin{bmatrix}
\Gamma_1 \, T_1 \, \tau_1 \\
0 \\
0 \\
c_1^2 \end{bmatrix} +\begin{bmatrix}
0 \\
u\big( \dot{U}_2 -\dot{U}_3 \big) \\
v\big( \dot{U}_2 -\dot{U}_3 \big) \\
0 \end{bmatrix} \, ,
$$
for some scalar $\dot{\mu}$, where $\dot{U}_3$ satisfies \eqref{expressionU3.} and therefore belongs to $E_-(0,1)$ (because the 
vector ${\mathcal R}(\underline{U}_1)$ belongs to $E_-(0,1)$ as seen in the first part of the proof of Theorem \ref{thm1}).

It appears from \eqref{expressionderivee1} that the partial derivative $\partial_\eps \delta (0,\underline{U}_0)$ vanishes if and only if 
the velocity ${\bf u} (\dot{U}_3 -\dot{U}_2)$ is parallel to ${\bf u}_1$. Hence the partial derivative $\partial_\eps \delta (0,\underline{U}_0)$ 
vanishes if and only if the vector $\dot{U}_2 -\dot{U}_3$ belongs to the vector space
\begin{equation*}
\text{\rm Span } \left\{ 
\begin{bmatrix}
0 \\
\overline{u} \\
v_1 \\
0 \end{bmatrix} \, , \, \begin{bmatrix}
\Gamma_1 \, T_1 \, \tau_1 \\
0 \\
0 \\
c_1^2 \end{bmatrix} \right\} \, = \, E_0(0,1) \, .
\end{equation*}
Using \eqref{expressionU2.} and \eqref{expressionU3.}, this means that the partial derivative $\partial_\eps \delta (0,\underline{U}_0)$ 
vanishes if and only if the solution $\dot{U}_2$ to the linear system
\begin{equation*}
{\rm d}f_2(\underline{U}_1) \, \dot{U}_2 \, = \, -\big( f_1(\underline{U}_1) -f_1(\underline{U}_0) \big) \, ,
\end{equation*}
belongs to $E^s(0,1)$, which, as we have already seen, occurs if and only if $\overline{u} =-V$. We therefore impose $\overline{u}=-V$, 
which yields $\partial_\eps \delta (0,\underline{U}_0)=0$. It now only remains to verify the invertibility condition $\partial^2_{\eps,u} \delta 
(0,\underline{U}_0) \neq 0$.

\paragraph{The invertibility condition.} We shall need in the calculations below the precise expression of several quantities that have 
arisen earlier. Consistently with the notation in \cite[page 124]{MR2}, we define
$$
\beta \, := \, \cos \Psi_0 \, ,
$$
where the angle $\Psi_0$ is defined by \eqref{determinangle} (we recall that $\overline{u}=-V$ so we have $\overline{u}^2+v_1^2>c_1^2$). 
The expression found for $\beta$ in \cite{MR2} is exactly the one which we have recalled in \eqref{defbetac}. One can check that it can be 
equivalently defined by \eqref{determinangle} provided that $\overline{u}$ coincides with $-V$. (The significance of the parameter $\Phi$ in 
\cite{MR2}, which is defined as a root to \eqref{eqPhi}, is less clear.) In order to factorize many expressions below, we also introduce the notation
\begin{equation}
\label{defomega}
\Upsilon \, := \, \dfrac{\overline{u}}{c_1} \, .
\end{equation}
Recalling \eqref{determinangle}, we have
$$
\beta \, = \, \dfrac{M_1+\Upsilon \, \sqrt{\Upsilon^2+M_1^2-1}}{\Upsilon^2+M_1^2} \, ,\quad 
\sqrt{1-\beta^2} \, = \, \dfrac{-\Upsilon \, +M_1 \, \sqrt{\Upsilon^2+M_1^2-1}}{\Upsilon^2+M_1^2} \, ,
$$
which, in particular, implies (compare with the definition of $c_\star$ in \eqref{defbetac})
\begin{equation}
\label{expromega}
\Upsilon \, = \, -\dfrac{1-M_1 \, \beta}{\sqrt{1-\beta^2}} \, ,\qquad \sqrt{\Upsilon^2+M_1^2-1} \, = \, \dfrac{M_1-\beta}{\sqrt{1-\beta^2}} \, .
\end{equation}

Since we have $\overline{u}=-V$, we know that the Lopatinskii determinant $\Delta (\overline{u},0,1)$ vanishes. As recalled earlier, this means 
that the linear system \eqref{linearRH} (with $z=0$, $\eta=1$ and $\chi=1$) has a solution in $E^s(0,1)$. For later purposes, we shall need the 
expression of that solution, which is nothing but the vector $\dot{U}_2$ defined in \eqref{expressionU2.}. The solution to \eqref{linearRH} with 
$z=0$, $\eta=1$, and $\chi=1$, is given by:
\begin{equation}
\label{solutionRH}
\begin{cases}
\dot{\tau} \, = \, \left( 2+\Gamma_1 \, \left( 1-\dfrac{\tau_0}{\tau_1} \right) \right) \, 
\dfrac{v_1^2 \, \overline{u}}{v_0 \, (c_1^2-v_1^2)} \, (\tau_1-\tau_0) \, ,& \\
\dot{u} \, = \, v_1-v_0 \, , & \\
\dot{v} \, = \, \dfrac{\overline{u}}{c_1^2-v_1^2} \, \left( \dfrac{\tau_1}{\tau_0} -1 \right) \, \left( 
c_1^2 +v_1^2 +v_1^2 \, \Gamma_1 \, \left( 1-\dfrac{\tau_0}{\tau_1} \right) \right) \, ,& \\[2ex]
\dot{s} \, = \, \dfrac{{\bf j}^2}{T_1} \, \dfrac{\overline{u}}{v_0} \, (\tau_1-\tau_0)^2 \, . &
\end{cases}
\end{equation}
The mass flux ${\bf j}$ is defined in \eqref{RH}. The vector $\dot{U}_2$ is decomposed on $E^s(0,1)$ as follows:
\begin{equation}
\label{decompositionv.}
\dot{U}_2 \, = \, \alpha_0 \, \begin{bmatrix}
0 \\
\overline{u} \\
v_1 \\
0 \end{bmatrix} \, + \, \mu_0 \, \begin{bmatrix}
\Gamma_1 \, T_1 \, \tau_1 \\
0 \\
0 \\
c_1^2 \end{bmatrix} \, + \, \alpha_- \, \begin{bmatrix}
\tau_1 \\
c_1 \, \sin \Psi_0 \\
-c_1 \, \cos \Psi_0 \\
0 \end{bmatrix} \, ,
\end{equation}
with
\begin{equation}
\label{defxz}
\alpha_0 \, := \, -\dfrac{\overline{u}}{v_0} \, \dfrac{(v_1-v_0)^2}{\overline{u}^2 +v_1^2} \, ,\quad 
\alpha_- \, := \, \dfrac{v_1-v_0}{c_1 \, \sin \Psi_0} \, \dfrac{v_1}{v_0} \, 
\dfrac{\overline{u}^2 +v_0 \, v_1}{\overline{u}^2 +v_1^2} \, .
\end{equation}
The coefficient $\mu_0$ equals $\dot{s}/c_1^2$, with $\dot{s}$ given in \eqref{solutionRH}, but its expression will not be relevant 
in the subsequent analysis. Recalling the definition \eqref{defnu} of the positive parameter $\nu$ and using the expression 
\eqref{expromega} of $\Upsilon$, we can rewrite the coefficients $\alpha_0$ and $\alpha_-$ in \eqref{defxz} as
\begin{equation}
\label{defalpha0-}
\alpha_0 \, = \, -\dfrac{M_1\, \nu^2}{1+\nu} \, \dfrac{(1-M_1 \, \beta) \, \sqrt{1-\beta^2}}{1+M_1^2 -2\, M_1 \, \beta} \, ,\quad 
\alpha_- \, = \, -\dfrac{M_1\, \nu}{1+\nu} \, \left( \dfrac{1}{\sqrt{1-\beta^2}} 
+\dfrac{M_1^2 \, \nu \, \sqrt{1-\beta^2}}{1+M_1^2 -2\, M_1 \, \beta} \right) \, .
\end{equation}
For future use, we note that the vectors $\dot{U}_2,\dot{U}_3$ in \eqref{defU2.U3.} satisfy
\begin{equation}
\label{relationU23}
{\bf u}(\dot{U}_2 -\dot{U}_3) \, = \, \alpha_0 \, {\bf u}_1 \, ,\quad \lambda(0,\underline{U}_0) \, = \, \alpha_- \, ,
\end{equation}
where the coefficients $\alpha_0,\alpha_-$ are given in \eqref{defalpha0-}.
\bigskip

Let us now compute $\partial_{\eps,u}^2 \delta (0,\underline{U}_0)$ by differentiating \eqref{defdelta}. Observing that 
$\partial_u \U_1(\underline{U}_0)$ equals $(0,1,0,0)$, we obtain
\begin{align*}
\partial_{\eps,u}^2 \delta (0,\underline{U}_0) \, = \, & \, \det \, \left| \begin{bmatrix}
1 \\
0 \end{bmatrix} \, , \, {\bf u} (\dot{U}_3-\dot{U}_2) \right| 
-\det \big| {\bf u}_1 \, , \, {\bf u} (\partial^2_{\eps,u} \U_2 (0,\underline{U}_0)) \big| \\
& \, +\det \big| {\bf u}_1 \, , \, \partial_u \lambda (0,\underline{U}_0) \, {\bf u} ({\mathcal R}(\underline{U}_1)) \big| 
+\det \big| {\bf u}_1 \, , \, \lambda (0,\underline{U}_0) \, {\bf u} (\partial_u {\mathcal R}(\underline{U}_1)) \big| \\
= \, & \, -\alpha_0 \, v_1 -\det \big| {\bf u}_1 \, , \, {\bf u} (\partial^2_{\eps,u} \U_2 (0,\underline{U}_0)) \big| \\
& \, -c_1 \, (\overline{u} \, \cos \Psi_0 +v_1 \, \sin \Psi_0) \, \partial_u \lambda (0,\underline{U}_0) 
+\alpha_- \, \det \big| {\bf u}_1 \, , \, {\bf u} (\partial_u {\mathcal R}(\underline{U}_1)) \big| \, ,
\end{align*}
where we have used \eqref{relationU23} and the expression \eqref{conventionR} of ${\mathcal R}(\underline{U}_1)$. Using now 
\eqref{determinangle}, we have already simplified the expression of $\partial_{\eps,u}^2 \delta (0,\underline{U}_0)$ into:
\begin{multline*}
\dfrac{1}{c_1} \, \partial_{\eps,u}^2 \delta (0,\underline{U}_0) \, = \, 
M_1 \, \alpha_0 -\dfrac{1}{c_1} \, \det \big| {\bf u}_1 \, , \, {\bf u} (\partial^2_{\eps,u} \U_2 (0,\underline{U}_0)) \big| \\
-c_1 \, \sqrt{\Upsilon^2 +M_1^2 -1} \, \partial_u \lambda (0,\underline{U}_0) 
+\dfrac{\alpha_-}{c_1} \, \det \big| {\bf u}_1 \, , \, {\bf u} (\partial_u {\mathcal R}(\underline{U}_1)) \big| \, ,
\end{multline*}
that is, using \eqref{expromega},
\begin{multline}
\label{calcul1}
\dfrac{1}{c_1} \, \partial_{\eps,u}^2 \delta (0,\underline{U}_0) \\
= \, M_1 \, \alpha_0 -\dfrac{1}{c_1} \, \det \big| {\bf u}_1 \, , \, {\bf u} (\partial^2_{\eps,u} \U_2 (0,\underline{U}_0)) \big| 
-\dfrac{M_1-\beta}{\sqrt{1-\beta^2}} \, c_1 \, \partial_u \lambda (0,\underline{U}_0) 
+\dfrac{\alpha_-}{c_1} \, \det \big| {\bf u}_1 \, , \, {\bf u} (\partial_u {\mathcal R}(\underline{U}_1)) \big| \, .
\end{multline}
There are three still undetermined quantities on the right hand side of \eqref{calcul1}: the scalar $\partial_u \lambda (0,\underline{U}_0)$, 
and the vectors $\partial_u {\mathcal R}(\underline{U}_1)$, $\partial^2_{\eps,u} \U_2 (0,\underline{U}_0)$.
\bigskip

The vector $\partial^2_{\eps,u} \U_2 (0,\underline{U}_0)$ is obtained by differentiating the second equation in \eqref{defU12} with respect 
to $\eps$ and $u$, and by therefore solving the linear system:
$$
{\rm d}f_2(\underline{U}_1) \, \partial^2_{\eps,u} \U_2 (0,\underline{U}_0) \, = \, -\partial_u {\rm d}f_2(\underline{U}_1) \, \dot{U}_2 
+\partial_u f_1(\underline{U}_0) -\partial_u f_1(\underline{U}_1) \, .
$$
This means that $\partial^2_{\eps,u} \U_2 (0,\underline{U}_0)$ is a solution to the linear system
$$
{\rm d}f_2(\underline{U}_1) \, \partial^2_{\eps,u} \U_2 (0,\underline{U}_0) \, = \, \begin{bmatrix}
\rho_0 -\rho_1 \\
(\rho_0 -\rho_1) \, \overline{u} \\
0 \\
\rho_0 \, \left( \dfrac{1}{2} \, |{\bf u}_0|^2 +e_0 \right) -\rho_1 \, \left( \dfrac{1}{2} \, |{\bf u}_1|^2 +e_1 \right) \end{bmatrix} \, ,
$$
and we eventually obtain the very simple expression
\begin{equation*}
\partial^2_{\eps,u} \U_2 (0,\underline{U}_0) =\dfrac{1}{\overline{u}} \, \begin{bmatrix}
\, \dot{\tau} \, \\
0 \\
\, \dot{v} \, \\
\, \dot{s} \, \end{bmatrix} \, ,\quad 
\dfrac{\dot{v}}{c_1} \, = \, -M_1 \, \alpha_0 +\beta \, \alpha_- \, ,
\end{equation*}
where $\dot{\tau},\dot{v},\dot{s}$ are the first, third and fourth coordinates of the vector $\dot{U}_2$ and are given in \eqref{solutionRH} 
or \eqref{decompositionv.}. This already simplifies \eqref{calcul1} into
\begin{equation}
\label{calcul1-1}
\dfrac{1}{c_1} \, \partial_{\eps,u}^2 \delta (0,\underline{U}_0) \, = \, 2 \, M_1 \, \alpha_0 +\beta \, \alpha_- 
-\dfrac{M_1-\beta}{\sqrt{1-\beta^2}} \, c_1 \, \partial_u \lambda (0,\underline{U}_0) 
+\dfrac{\alpha_-}{c_1} \, \det \big| {\bf u}_1 \, , \, {\bf u} (\partial_u {\mathcal R}(\underline{U}_1)) \big| \, .
\end{equation}
\bigskip

Differentiating now the pressure equality \eqref{relationlambda} with respect to $\eps$ and $u$, we get
$$
-\dfrac{c_1^2}{\tau_1^2} \, \tau(\partial^2_{\eps,u} \U_2 (0,\underline{U}_0)) +\dfrac{c_1^2}{\tau_1} \, \partial_u \lambda (0,\underline{U}_0) 
+\dfrac{\Gamma_1 \, T_1}{\tau_1} \, s(\partial^2_{\eps,u} \U_2 (0,\underline{U}_0)) \, = \, 0 \, ,
$$
or equivalently (here we use the expression of $\partial^2_{\eps,u} \U_2 (0,\underline{U}_0)$ and \eqref{linearRH} with $z=0$, $\eta=1$ and 
$\chi=1$)
\begin{equation*}
c_1 \, \partial_u \lambda (0,\underline{U}_0) \, = \, \dfrac{1}{\overline{u} \, c_1} \, \left( 
\dfrac{c_1^2}{\tau_1} \, \dot{\tau} -\Gamma_1 \, T_1 \, \dot{s} \right) \, = \, 
\dfrac{v_1}{\overline{u} \, c_1} \, \dot{v} -\dfrac{v_1}{c_1} \, \left( 1 -\dfrac{\tau_1}{\tau_0} \right) \, = \, 
\dfrac{1}{\Upsilon} \, \big( M_1^2 \, \alpha_0 +M_1 \, \beta \, \alpha_- \big) +\dfrac{M_1 \, \nu}{1+\nu} \, .
\end{equation*}
This simplifies \eqref{calcul1-1} into
\begin{multline*}
\dfrac{1}{c_1} \, \partial_{\eps,u}^2 \delta (0,\underline{U}_0) \, = \, 2 \, M_1 \, \alpha_0 +\beta \, \alpha_- 
+\dfrac{M_1-\beta}{1-M_1 \, \beta} \, \big( M_1^2 \, \alpha_0 +M_1 \, \beta \, \alpha_- \big) 
-\dfrac{M_1 \, \nu \, (M_1-\beta)}{(1+\nu) \, \sqrt{1-\beta^2}} \\
+\dfrac{\alpha_-}{c_1} \, \det \big| {\bf u}_1 \, , \, {\bf u} (\partial_u {\mathcal R}(\underline{U}_1)) \big| \, ,
\end{multline*}
that is, after a little bit of factorizations
\begin{multline}
\label{calcul1-2}
\dfrac{1}{c_1} \, \partial_{\eps,u}^2 \delta (0,\underline{U}_0) \, = \, \dfrac{M_1 \, (2+M_1^2)-3 \, M_1^2 \, \beta}{1-M_1 \, \beta} \, \alpha_0 
+\dfrac{1+M_1^2-2 \, M_1 \, \beta}{1-M_1 \, \beta} \, \beta \, \alpha_- \\
-\dfrac{M_1 \, \nu \, (M_1-\beta)}{(1+\nu) \, \sqrt{1-\beta^2}} 
+\dfrac{\alpha_-}{c_1} \, \det \big| {\bf u}_1 \, , \, {\bf u} (\partial_u {\mathcal R}(\underline{U}_1)) \big| \, .
\end{multline}

Using \eqref{anglequelconque}, we now differentiate the definition \eqref{conventionR} with respect to $u$ and get
\begin{equation*}
\partial_u {\mathcal R} (\underline{U}_1) \, = \, -\dfrac{c_1 \, \sin\Psi_0}{\sqrt{\overline{u}^2 +v_1^2 -c_1^2}} 
\, \begin{bmatrix}
0 \\
\cos \Psi_0 \\
\sin \Psi_0 \\
0 \end{bmatrix} \, = \, \dfrac{1-\beta^2}{M_1-\beta} \, \begin{bmatrix}
0 \\
\beta \\
- \sqrt{1-\beta^2} \\
0 \end{bmatrix} \, .
\end{equation*}
Incorporating this last expression in \eqref{calcul1-2}, we obtain
\begin{multline}
\label{calcul2}
\dfrac{1}{c_1} \, \partial_{\eps,u}^2 \delta (0,\underline{U}_0) \, = \, \dfrac{M_1 \, (2+M_1^2)-3 \, M_1^2 \, \beta}{1-M_1 \, \beta} \, \alpha_0 
+\dfrac{1+M_1^2-2 \, M_1 \, \beta}{1-M_1 \, \beta} \, \beta \, \alpha_- \\
+\dfrac{1-\beta^2}{M_1-\beta} \, \alpha_- -\dfrac{M_1 \, \nu \, (M_1-\beta)}{(1+\nu) \, \sqrt{1-\beta^2}} \, .
\end{multline}

We now use the expressions of $\alpha_0$ and $\alpha_-$ in \eqref{defalpha0-} and derive the final expression
\begin{multline*}
\dfrac{1+\nu}{\nu \, v_1 \, \sqrt{1-\beta^2}} \, \partial_{\eps,u}^2 \delta (0,\underline{U}_0) \\
= \, \dfrac{1+M_1^2 -2\, M_1 \, \beta}{(M_1 -\beta) \, (1-M_1 \, \beta)} 
+\nu \, M_1 \, \dfrac{M_1\, (3+M_1^2) \, \beta^2 -2\, (1+3\, M_1^2) \, \beta +M_1\, (3+M_1^2)}
{(M_1 -\beta) \, (1-M_1 \, \beta) \, (1+M_1^2 -2\, M_1 \, \beta)} \, .
\end{multline*}
Recalling that $M_1$ belongs to $(0,1)$ and that $\beta$ belongs to $(-1,M_1)$, we see that up to a nonzero multiplicative factor, 
$\partial_{\eps,u}^2 \delta (0,\underline{U}_0)$ can we written as $\Omega_0 +\nu \, M_1 \, \Omega_1$ where both $\Omega_0$ 
and $\Omega_1$ are positive quantities. Consequently, we have $\partial_{\eps,u}^2 \delta (0,\underline{U}_0) \neq 0$ and we can 
apply the implicit function Theorem to the function $\widetilde{\delta}$. We can therefore construct a smooth function $U_0(\eps)$ of 
the form \eqref{formU0eps} such that $\delta(\eps,U_0(\eps))=0$ for all $\eps$ close to $0$.

Since the velocities of the two states $\U_2(\eps,U_0(\eps))$ and $\U_3(\eps \, \lambda(\eps,\U_1(U_0(\eps))),\U_1(U_0(\eps)))$ are 
collinear and close to ${\bf u}_1$, we can determine the angle $\Phi(\eps)$, close to $\Phi_0$, as their common argument and therefore 
determine the location of the contact discontinuity ${\bf CD}(\eps)$.
\bigskip

$\bullet$ \underline{Step 5. Conclusion.}

It remains to verify that the corresponding four wave pattern which we have just constructed defines a Mach stem for any \emph{positive} 
$\eps$. For the sake of clarity, let us denote
$$
U_1 (\eps) \, := \, \U_1 (U_0(\eps)) \, ,\quad U_2 (\eps) \, := \, \U_2 (\eps,U_0(\eps)) \, ,\quad 
U_3 (\eps) \, := \, \U_3 (\eps \, \lambda (\eps,U_0(\eps)),U_1 (\eps)) \, .
$$
We also define the angle
$$
\Psi(\eps) \, := \, \Psi (U_1 (\eps),U_3(\eps)) \, ,
$$
which represents the location of the discontinuity between $U_1 (\eps)$ and $U_3(\eps)$. All these functions are defined on a common 
interval $[-\eps_0,\eps_0]$ and are smooth functions of $\eps$ on that interval.

Recalling that the vector $U_0(\eps)$ is of the form \eqref{formU0eps}, we compute
\begin{align*}
U_0'(0) \, & \, = \, \big( 0,u'(0),0,0 \big) \, ,\qquad U_1'(0) \, = \, \big( 0,u'(0),0,0 \big) \, ,\\
U_3'(0) \, & \, = \, \big( \alpha_- \, \tau_1,\alpha_- \, c_1 \, \sin \Psi_0 +u'(0),-\alpha_- \, c_1 \, \cos \Psi_0,0 \big) \, ,
\end{align*}
where we have used the relation $\partial_u \U_1(\underline{U}_0)=(0,1,0,0)$, \eqref{relationU23}, \eqref{conventionR} and the properties 
of $\U_3$.

We know that the states $U_0(\eps)$ and $U_1(\eps)$ satisfy the Rankine-Hugoniot jump relations across $\bfS_2(\eps) = \{ x_2=0 \}$. 
Similarly $U_0(\eps)$ and $U_2(\eps)$ satisfy the Rankine-Hugoniot jump relations across $\bfS_1(\eps) = \{ \sin \eps \, x_1 +\cos \eps 
\, x_2 =0 \}$. Since the reference state \eqref{shock} satisfies \eqref{RH}, the discontinuities $\bfS_1(\eps)$ and $\bfS_2(\eps)$ are 
necessarily shock waves with $U_0(\eps)$ being the state ahead of the shock. We also know from our previous construction that the 
states $U_2(\eps)$ and $U_3(\eps)$ are connected by a contact discontinuity, which we denote ${\bf CD}(\eps)$, whose angle $\Phi (\eps)$ 
is close to $\Phi_0$. The causality conditions are easily verified since we have
$$
{\bf u}_1 \cdot (\cos \Phi_0,\sin \Phi_0) \, = \, |{\bf u}_1| \, > \, 0 \, ,\quad 
{\bf u}_1 \cdot (\cos \Psi_0,\sin \Psi_0) \, = \, \sqrt{\overline{u}^2+v_1^2-c_1^2} \, > \, 0 \, ,
$$
by the definition of the angles $\Phi_0$ and $\Psi_0$. By continuity of the states $U_1(\eps)$, $U_2(\eps)$, $U_3(\eps)$ with respect to 
$\eps$, we thus have
$$
{\bf u}(U_2(\eps)) \cdot (\cos \Phi(\eps),\sin \Phi(\eps)) \, > \, 0 \, ,\quad 
{\bf u}(U_1(\eps)) \cdot (\cos \Psi(\eps),\sin \Psi(\eps)) \, > \, 0 \, ,
$$
for all $\eps$ (up to restricting the interval $[-\eps_0,\eps_0]$).

The only remaining fact to clarify is to determine whether the discontinuity $\bfS_3(\eps)$ that connects the states $U_1(\eps)$ and 
$U_3(\eps)$ is a shock wave. This is where the sign condition on $\eps$ will arise, as in \cite{MR2}. Let us introduce the function
\begin{align*}
\F(\Psi,U) \, := \, & \, \det \Big( -\sin \Psi \, A_1(U,U) +\cos \Psi \, A_2(U,U) \Big) \\
= \, & \, \det P(U)^{-1} \, \big( -\sin \Psi \, u +\cos \Psi \, v \big)^2 \, \big( -\sin \Psi \, u +\cos \Psi \, v-c \big) \, 
\big( -\sin \Psi \, u +\cos \Psi \, v+c \big) \, .
\end{align*}
This function vanishes at $(\Psi_0,\underline{U}_1)$ because of the factor $-\sin \Psi \, u +\cos \Psi \, v+c$. All other factors are nonzero 
at $(\Psi_0,\underline{U}_1)$. Differentiating with respect to $\eps$ the relation
$$
\det \Big( -\sin \Psi(\eps) \, A_1(U_1(\eps),U_3(\eps)) +\cos \Psi(\eps) \, A_2(U_1(\eps),U_3(\eps)) \Big) \, = \, 0 \, ,
$$
and using the symmetry of $A_1,A_2$ with respect to $U_1,U_3$, see \eqref{defA1A2}, we end up with the relation
\begin{equation*}
\partial_\Psi \F (\Psi_0,\underline{U}_1) \, \Psi'(0) 
+\dfrac{1}{2} \, {\rm d}_U \F (\Psi_0,\underline{U}_1) \cdot \big( U_1'(0) +U_3'(0) \big) \, = \, 0 \, .
\end{equation*}
Using the decomposition of $\F$ as a product and the expressions of $U_1'(0)$ and $U_3'(0)$, we obtain the relation
\begin{equation}
\label{relationPsi'0}
\sqrt{\overline{u}^2 +v_1^2 -c_1^2} \, \Psi'(0) +\sin \Psi_0 \, u'(0) \, = \, -\dfrac{1}{2} \, \alpha_- \, c_1 \, {\mathcal G}_1 \, ,\quad 
{\mathcal G}_1 \, := \, -\dfrac{\tau_1}{2} \, \dfrac{\partial_{\tau \tau \tau}^3 e (\tau_1,s_1)}{\partial_{\tau \tau}^2 e (\tau_1,s_1)} \, ,
\end{equation}
which connects the derivative $\Psi'(0)$ of the angle $\Psi(\eps)$ with the first order variation $u'(0)$ of the state $U_0(\eps)$. 
Let us observe that because of our assumptions on the equation of state, ${\mathcal G}_1$ is a positive quantity (this quantity 
is a measure of the genuine nonlinearity of the characteristic fields associated with the acoustic waves, see \cite{mplohr}).

We can now verify that the discontinuity ${\bf S}_3(\eps)$ is a shock wave for any sufficiently small $\eps>0$. Indeed, the only 
thing to verify is that Lax shock inequalities
\begin{equation*}
\dfrac{{\bf u}(U_1(\eps)) \cdot (\sin \Psi(\eps),-\cos \Psi(\eps))}{c(U_1(\eps))} \, > \, 1 \, > \, 
\dfrac{{\bf u}(U_3(\eps)) \cdot (\sin \Psi(\eps),-\cos \Psi(\eps))}{c(U_3(\eps))} \, ,
\end{equation*}
are satisfied for any sufficiently small $\eps>0$. Using the expressions of $U_1'(0)$ and $U_3'(0)$, we compute
\begin{align*}
\dfrac{{\rm d}}{{\rm d} \eps} \Big( {\bf u}(U_1(\eps)) \cdot (\sin \Psi(\eps),-\cos \Psi(\eps)) -c(U_1(\eps)) \Big) \Big|_{\eps=0} 
\, = \, & \, \sqrt{\overline{u}^2 +v_1^2 -c_1^2} \, \Psi'(0) +\sin \Psi_0 \, u'(0) \\
= \, & \, -\dfrac{1}{2} \, \alpha_- \, c_1 \, {\mathcal G}_1 \, ,
\end{align*}
where we have used \eqref{relationPsi'0}, and we also compute\footnote{Here we use the relation $\partial c/\partial \tau=(1-{\mathcal G}) 
\, c/\tau$.}
\begin{align*}
\dfrac{{\rm d}}{{\rm d} \eps} \Big( {\bf u}(U_3(\eps)) \cdot (\sin \Psi(\eps),-\cos \Psi(\eps)) -c(U_3(\eps)) \Big) \Big|_{\eps=0} 
\, = \, & \, \sqrt{\overline{u}^2 +v_1^2 -c_1^2} \, \Psi'(0) +\sin \Psi_0 \, u'(0) +\alpha_- \, c_1 \, {\mathcal G}_1\\
= \, & \, \dfrac{1}{2} \, \alpha_- \, c_1 \, {\mathcal G}_1 \, ,
\end{align*}
where we have used \eqref{relationPsi'0} again. We observe from the expression \eqref{defalpha0-} that $\alpha_-$ is negative (recall 
that the parameter $\nu$ in \eqref{defnu} is positive), and since ${\mathcal G}_1$ is positive, we thus have
\begin{align*}
\dfrac{{\rm d}}{{\rm d} \eps} \Big( {\bf u}(U_1(\eps)) \cdot (\sin \Psi(\eps),-\cos \Psi(\eps)) -c(U_1(\eps)) \Big) \Big|_{\eps=0} 
\, & \, > \, 0 \, ,\\
\dfrac{{\rm d}}{{\rm d} \eps} \Big( {\bf u}(U_3(\eps)) \cdot (\sin \Psi(\eps),-\cos \Psi(\eps)) -c(U_3(\eps)) \Big) \Big|_{\eps=0} 
\, & \, < \, 0 \, ,
\end{align*}
which implies that Lax shock inequalities are satisfied for the discontinuity ${\bf S}_3(\eps)$ for any sufficiently small $\eps>0$. 
This completes the proof of Theorem \ref{thm1}.
\end{proof}

\bibliographystyle{alpha}
\bibliography{MachStems}
\end{document}